\documentclass{emsprocart}

\contact[roggi@raunvis.hi.is]{R\"ognvaldur G.~M\"oller, Science Institute, University of Iceland, Dunhaga 3, IS-107 Reykjavík, Iceland}

\newtheorem{theorem}{Theorem}[section]
\newtheorem{corollary}[theorem]{Corollary}
\newtheorem{lemma}[theorem]{Lemma}
\newtheorem{proposition}[theorem]{Proposition}
\newtheorem{conjecture}[theorem]{Conjecture}

\theoremstyle{definition}
\newtheorem{definition}[theorem]{Definition}

\newcommand{\Q}{\mbox{\bf Q}}

\newcommand{\Z}{\mbox{\bf Z}}
\newcommand{\s}{\mbox{\bf s}}

\newcommand{\nin}{\mbox{$\ \not\in\ $}}
\newcommand{\sym}{\mbox{${\rm Sym}$}}
\newcommand{\aut}{\mbox{${\rm Aut}$}}
\newcommand{\Sym}{\mbox{${\rm Sym}$}}
\newcommand{\Cay}{\mbox{${\rm Cay}$}}

\newcommand{\autga}[0]{\mbox{${\rm Aut}(\Gamma)$}}

\newcommand{\anc}{\mbox{${\rm anc}$}}
\newcommand{\des}{\mbox{${\rm desc}$}}
\newcommand{\inn}{\mbox{${\rm in}$}}
\newcommand{\out}{\mbox{${\rm out}$}}
\newcommand{\supp}{\mbox{${\rm supp}\,$}}

\title[Graphs, permutations, and topological groups]{Graphs, permutations and topological groups}

\author[R\"ognvaldur G.~M\"oller]{R\"ognvaldur G.~M\"oller}

\begin{document}

\begin{abstract}
Various connections between the theory of permutation groups and the theory of topological groups are described.  These connections are applied in permutation group theory and in the structure theory of topological groups.
\end{abstract}



\maketitle

\section*{Introduction}

The aim of this paper is to discuss various links between  permutation groups, graphs and topological groups.  The action of a group on a set can be used to define a topology on the group, called the {\em permutation topology}.  The earliest references for this topology are the paper \cite{Maurer1955} by Maurer and the paper \cite{KarrassSolitar1956} by Karrass and Solitar.  This topology opens up the possibility of applying concepts and results  from the theory of topological groups in permutation group theory.  One can also go the other way and apply ideas from permutation group theory to problems about topological groups.  In particular, some simple constructions of graphs, that are commonly used in permutation group theory, can be applied.

In the first section we discuss the languages we use, when working with graphs and permutation groups.

In Section 2 we look at the definition of the permutation topology and
consider applications of the theory of topological groups to questions
about permutation groups.  The main result in this section is a
theorem of Schlichting from \cite{Schlichting1980}.  This is a theorem
about permutation groups, but Schlichting's proof uses notions from
functional analysis and result of Iwasawa \cite{Iwasawa1951} about
topological groups.  Here we present a proof using concepts from
permutation group theory and  Iwasawa's Theorem.

In the third and fourth sections we discuss applications of techniques and ideas from permutation groups theory and graph theory to the theory of topological groups.

In Section 3 Willis' structure theory of totally disconnected locally compact groups is in the limelight.  Willis' paper \cite{Willis1994} helped spark a new interest in totally disconnected locally compact groups.  Later work by Willis and others has shown that the concepts of the theory have many applications and are open to various interpretations.  Most of the material in this section comes from the paper \cite{Moller2002}, where the basics of Willis' theory are given a graph theoretic interpretation.

In the fourth and last section we discuss an analogue of a Cayley graph, called a rough Cayley graph, that one can construct for a compactly generated, totally disconnected, group. The rough Cayley graph is defined in Section~\ref{SDefRough}.  In that section, it is also shown that this graph is a quasi-isometry invariant of the group.   In the latter parts of Section 4, it is shown how one uses rough Cayley graphs, by developing an analogue of the theory of  ends of groups and the theory of groups with polynomial growth.

There are various other topics, that should be discussed in  a survey like this.  The study of random walks on groups and graphs is another place where graphs, permutations and topological groups meet.  The book by Woess \cite{Woess2001} is an excellent introduction to this field.  Another meeting place for graphs, permutations and topological groups is the theory of groups acting on trees.  In particular, one could mention the theory of harmonic analysis and representation theory of groups acting on trees, see the book by  Fig{\`a}-Talamanca and Nebbia \cite{Figa-TalamancaNebbia1991} and the theory of tree lattices, see the book by Bass and Lubotzky \cite{BassLubotzky2001}.
Then there is the topic of generic elements and subgroups, see the papers [7, 8] by Bhattacharjee and the paper \cite{AbertGlasner2008} by Abert and Glasner.  And then I have not even mentioned the manifold appearances of our trio of graphs, permutations and topology in model theory.  Describing all these topics would have meant a book length paper.

\section{Languages for graphs and permutation groups}

\subsection{A language for graphs}\label{SGraphs}
We will discuss both {\em undirected graphs}, or just {\em
  graphs}, and {\em directed graphs},  called in this paper {\em
  digraphs}.

Our undirected graphs are without loops and multiple edges.  Thus
one can think of a (undirected)
graph $\Gamma$ as an ordered pair $(V\Gamma, E\Gamma)$ where
$V\Gamma$ is a set and $E\Gamma$ is a set of two element subsets of
$V\Gamma$.  The
elements of $V\Gamma$ are called {\em vertices} and the elements of
$E\Gamma$ are called {\em edges}.

Vertices $\alpha$ and $\beta$ in a graph $\Gamma$ are said to be
{\em neighbours}, or {\em adjacent}, if $\{\alpha, \beta\}$ is an edge
in $\Gamma$.   The {\em valency} of a vertex is the number of its
neighbours. A graph, all  of whose vertices have finite valency,
is called {\em locally finite}.
For vertices $\alpha$ and $\beta$ of the graph, a \emph{walk} of length $n$
from $\alpha$ to $\beta$ is a sequence $\alpha=\alpha_0,\alpha_1,\ldots,\alpha_n=\beta$
of vertices, such that $\alpha_i$
and $\alpha_{i+1}$ are adjacent for $i=0, 1, \ldots, n-1$.  A walk, all
of whose vertices are distinct, is called a {\em path}.
A {\em ray} in a graph is a sequence $\alpha_0, \alpha_1, \ldots$ of
distinct vertices such that $\alpha_i$ is adjacent to $\alpha_{i+1}$
for all $i$.
 A graph
is connected if for any two vertices $\alpha$ and $\beta$ there is a
walk from $\alpha$ to $\beta$.  Let $d(\alpha,\beta)$ denote the length of
a shortest walk from a vertex $\alpha$ to a vertex $\beta$. For a connected
graph, the function $d$ is a metric on its set of vertices.
Let $A$ be a set of vertices of $\Gamma$.  The
{\em subgraph of $\Gamma$ spanned by $A$} is the graph whose vertex set is $A$,
and whose edge set is the set of all edges in $\Gamma$ whose end vertices are
both in $A$.  We say that a set of vertices $A$ is {\em
  connected}, if the subgraph spanned by $A$ is connected.  The {\em
  connected components}  (or just {\em components}) of a graph are the
maximal connected sets of vertices.

\medskip

We define a {\em digraph} $\Gamma$
to be a pair $(V\Gamma, E\Gamma)$ where
$V\Gamma$ is a set and $E\Gamma$ is a set of ordered pairs of distinct
elements of $V\Gamma$, i.e. $E\Gamma\subseteq (V\Gamma\times V\Gamma)
\setminus\{(\alpha, \alpha)\mid \alpha\in V\Gamma\}$.  An edge
$(\alpha, \beta)$ may be thought of as an \lq\lq arrow\rq\rq\ starting
in $\alpha$ and
ending in $\beta$.
The
{\em in-valency} of a vertex $\alpha$ in $\Gamma$ is the number of
edges of the type $(\beta, \alpha)$ (number of edges going \lq\lq in
to\rq\rq\ $\alpha$) and the {\em out-valency} of a vertex $\alpha$
is the number of edges $(\alpha, \beta)$ that go \lq\lq out of\rq\rq\
$\alpha$.
A digraph is said to be {\em locally finite} if every vertex has
finite in- and out-valencies.

One can \lq\lq forget\rq\rq\ the directions of edges in $\Gamma$ and define an undirected graph
$\overline{\Gamma}$ with the same vertex set and two vertices $\alpha$
and $\beta$ adjacent if and only if $(\alpha, \beta)$ or $(\beta,
\alpha)$ is an edge in the digraph $\Gamma$.

A {\em walk} in a digraph $\Gamma$ is a sequence $\alpha_0,
\ldots, \alpha_n$ such that $(\alpha_i, \alpha_{i+1})$ or
$(\alpha_{i+1}, \alpha_i)$ is an edge in $\Gamma$ for $i=0, \ldots,
n-1$.  (That is to say, $\alpha_0,
\ldots, \alpha_n$ is a walk in the undirected graph
$\overline{\Gamma}$ associated to $\Gamma$.)   An {\em arc}, more specifically an $n$-{\em arc},
is a sequence $\alpha_0,\ldots, \alpha_n$ of distinct vertices  such that
$(\alpha_i, \alpha_{i+1})$ is an edge in $\Gamma$ for $i=0, \ldots,
n-1$.  Arcs are sometimes referred to as directed paths.  The set of {\em descendants} of a vertex $\alpha$ is defined as
the set
$$\des(\alpha)=\{\beta\in V\Gamma\mid \mbox{there exists an arc
   from }\alpha\mbox{ to }\beta\}.$$
The set $\des_k(\alpha)$ is defined as the set of all vertices $\beta$
  such that the shortest arc from $\alpha$ to $\beta$ has
  length $k$.
The set of {\em ancestors} of a vertex $\alpha$, denoted
  $\anc(\alpha)$,  is defined as the set
  of all vertices $\beta$ such that $\alpha \in \des(\beta)$, i.e.\
  $\beta\in \anc(\alpha)$ if and only if there exists an arc
  from $\beta$ to $\alpha$.

\medskip

Finally, we review the definition of a \emph{Cayley graph} of a group.  Let $G$
be a group and $S$ a subset of $G$.  The (undirected) Cayley graph
$\Cay(G, S)$ of $G$ with
respect to $S$ has $G$ as the vertex set and $\{g, h\}$ is an edge if
$h=gs$ or $h=gs^{-1}$ for some $s$ in $S$.  The Cayley graph
$\Cay(G, S)$ is connected if and only if $S$ generates $G$.  The left regular
action of $G$ on itself gives a transitive action of $G$ as a group of
graph automorphisms on $\Cay(G, S)$.

It is common to define the Cayley graph of a group $G$ with respect to a subset
$S$ of $G$ as the digraph with vertex set $G$ and set of directed edges
the collection of all ordered pairs $(g,gs)$ for $g$ in $G$ and $s$ in $S$.
In these notes, however,  it is convenient to
think of Cayley graphs as being undirected.

\subsection{A language for permutation groups}\label{Sperm}
In this section $G$ is a group acting on a set $\Omega$.
The image of a point $\alpha\in \Omega$ under an
element $g\in G$ will be written $g\alpha$.
The group of all permutations of a set $\Omega$ is called the {\em
  symmetric group on} $\Omega$ and is denoted by $\Sym(\Omega)$.

The action of $G$ on $\Omega$ induces a natural homomorphisms
  $G\rightarrow \Sym(\Omega)$.
If this homomorphism is injective, then the group $G$ is said to act {\em faithfully} on $\Omega$.
A group $G$ which acts faithfully on a set $\Omega$ can be regarded as a
subgroup of $\Sym(\Omega)$, and then we say that $G$ is
 a {\em permutation group} on $\Omega$.

An action is said to be {\em
transitive}, if for every two points $\alpha, \beta\in \Omega$ there is
some element $g\in G$ such that $g\alpha=\beta$.  For a point $\alpha\in
\Omega$, the subgroup
$$G_\alpha=\{g\in G\mid g\alpha =\alpha\}$$
is called the {\em stabilizer} in $G$
of the point $\alpha$.
The {\em pointwise
stabilizer} $G_{(\Delta)}$
of a subset $\Delta$ of $\Omega$ is defined as the subgroup of all the elements in $G$ that fix
every element of $\Delta$, that is,
$$G_{(\Delta)}=\{g\in G\mid g\delta =\delta \mbox{ for every }\delta\in
\Delta\}=\bigcap_{\delta\in \Delta}G_\delta.$$
The {\em setwise stabilizer} $G_{\{\Delta\}}$ of $\Delta$ is defined as
the subgroup consisting of all elements of $G$ that leave $\Delta$
invariant, that is,
$$G_{\{\Delta\}}=\{g\in G\mid g\Delta=\Delta\}.$$
The {\em $G$-orbit}, or, simpler {\em orbit}, of a
point $\alpha$ is the set $\{g\alpha\mid g\in G\}$.

Suppose $U$ is a subgroup of $G$.  The group $G$ acts on the set
$G/U$ of right cosets of $U$, and this action is transitive.  The image of a
coset $hU$ under an element $g\in G$ is $(gh)U$.  Conversely, if $G$
acts transitively $\Omega$ and $\alpha$ is a point in
$\Omega$, then $\Omega$ can be identified with
$G/G_\alpha$.  Here \lq\lq identified\rq\rq\ means that there is a
bijective map $\theta:\Omega\rightarrow G/G_\alpha$ such that for
every $\omega\in \Omega$ and every element $g\in G$ we have
$\theta(g\omega)=g\theta(\omega)$.

The orbits of stabilizers of points in $\Omega$ are called {\em suborbits},
that is, the suborbits are sets of the form $G_\alpha\beta$ where $\alpha,
\beta\in \Omega$.  Orbits of $G$ on the set of ordered pairs of elements from
$\Omega$ are called {\it orbitals}.  When $G$ is transitive on
$\Omega$ one can, for a fixed point
$\alpha\in\Omega$, identify the suborbits of $G_\alpha$ with
the orbitals: the suborbit $G_\alpha\beta$ is identified with the
orbital $G(\alpha, \beta)$.
The
number of elements in a suborbit $G_\alpha\beta$, often called the
{\em length} of the suborbit, is given by the
index $|G_\alpha: G_\alpha\cap G_\beta|$.
The number of elements in the orbit $U\beta$ of a
subgroup $U$ is equal to
the index $|U:U\cap G_\beta|$.

Next, we define a digraph, whose vertex set is $\Omega$ and whose edge set is
the union of $G$-orbitals, and call it the \emph{directed orbital
graph of the action}.
The action of $G$ on its vertex set $\Omega$ induces an action of $G$ as a
group of automorphisms on the directed orbital graph $\Gamma$, because if
$(\alpha,\beta)$ is an edge in $\Gamma$ then $(g\alpha,g\beta)$ is
in the same orbital as $(\alpha,\beta)$ and therefore also an edge in $\Gamma$.
Similarly, we define the {\em undirected orbital graph} of a $G$-action on a
set $\Omega$ as a graph whose vertex set is $\Omega$ and whose edge set is the union of
$G$-orbits on the set of two element subsets of $\Omega$.

A {\em block of imprimitivity} for $G$ is a subset $\Delta$ of $\Omega$
such that for every $g\in G$, either $g\Delta=\Delta$ or $\Delta\cap
(g\Delta)=\emptyset$.  The existence of a non-trivial proper block of
imprimitivity $\Delta$ ({\em non-trivial} means that $|\Delta|>1$ and
{\em proper} means that $\Delta\neq \Omega$)
is equivalent to the existence of a
non-trivial proper $G$-invariant equivalence relation $\sim$ on
$\Omega$. If there is no non-trivial
proper $G$-invariant equivalence relation on $\Omega$ we say that $G$
acts primitively on $\Omega$.  In most books on permutation
groups it is shown that, if $G$ acts transitively on $\Omega$
then $G$ acts primitively on $\Omega$ if and only if $G_\alpha$ is a maximal
subgroup of $G$ for every $\alpha\in \Omega$.  Part of the proof of
this fact is to show that if $G_\alpha<H<G$ then $H\alpha$ is a
non-trivial proper block of imprimitivity.  A further useful fact
is that if $N$ is a
normal subgroup of $G$, then the orbits of $N$
on $\Omega$ are blocks of imprimitivity for $G$.

Recent books covering this material are
\cite{BMMN1997}, \cite{Cameron1999}  and \cite{DixonMortimer1996}.

\section{The permutation topology}

\subsection{Definition of the permutation topology}
Let $G$ be a group acting on a set $\Omega$.  The action
can be used to introduce a topology on $G$.  The topology of a
topological group is completely determined by a neighbourhood basis of
the identity element.  The {\em permutation topology} on $G$ is defined
by choosing as a neighbourhood basis of the identity the family of
pointwise stabilizers of finite subsets of $\Omega$, i.e.\ a neighbourhood
basis of the identity is given by the family of subgroups
\[\{G_{(\Phi)}\mid \Phi\mbox{ is a finite subset of }\Omega\}.\]
A sequence $(g_i)$ of elements in $G$ has an element
$g\in G$ as a limit if and only if for every
$\alpha\in \Omega$ there is a number $N$ (depending on $\alpha$) such
that $g_n\alpha =g\alpha$ for every $n\geq N$.
There are other ways to define the  permutation topology.
Think of $\Omega$ as having the discrete topology and elements of
$G$ as maps $\Omega\rightarrow \Omega$.
Then the permutation
topology is equal to the topology of pointwise convergence, and it is
also the same as the compact-open topology.

The basic idea is that two permutations $g$ and $h$ are \lq\lq
close\rq\rq\
to each other if they agree on \lq\lq many\rq\rq\  points.
If the set $\Omega$ is
countable, then the permutation topology can be defined by a metric, here is one way to do that.

\smallskip

{\it Enumerate
the points in
$\Omega$ as $\alpha_1, \alpha_2, \ldots$.  Take two elements $g, h\in G$.
Let $n$ be the smallest number such that $g\alpha_n\neq h\alpha_n$
or $g^{-1}\alpha_n\neq
h^{-1}\alpha_n$.  Set $d(g,h)=1/2^n$.   Then $d$ is a metric on $G$
that induces the permutation topology.}

\smallskip

From the definition of the permutation topology
we can immediately characterize open subgroups in $G$:

\smallskip
{\em A subgroup of $G$ is open if and only if it contains the pointwise
stabilizer of some finite set of points.}

\smallskip
Various properties of the action of $G$ on $\Omega$ are reflected by
properties of this topology on $G$.

\smallskip
{\em The permutation topology on $G$ is
Hausdorff if and only if the action of $G$ on $\Omega$ is faithful.
Moreover, $G$ is totally
disconnected if and only if the action is faithful.}

\bigskip
{\em Remark.}  In general we do not assume that our topological groups
are Hausdorff, but note that totally disconnected groups are always
Hausdorff.  

\smallskip

We say that a group $G$ acting on $\Omega$ is
a {\em closed} if it is image in $\sym(\Omega)$ is a closed subgroup, 
where $\sym(\Omega)$ has the permutation topology.  Closed
permutation groups can be characterized in the following way, see
\cite[Section 2.4]{Cameron1990}.

\begin{proposition}
A permutation group $G$ on a set $\Omega$ is closed if and only if $G$ is
the full automorphism group of some first order structure on $\Omega$.
\end{proposition}

\noindent
A first order structure on $\Omega$ is a collection of:
\begin{itemize}
\item constants that belong to $\Omega$,
\item functions defined on $\Omega$ and taking their values in $\Omega$,
\item relations defined on $\Omega$.
\end{itemize}
It is easy to show that the
automorphism group of such a structure is closed. To prove the converse,
one uses the concept of the {\em canonical relational structure} on $\Omega$
such that for $n=1, 2, \ldots$ we get one $n$-ary relation
for each orbit of $G$ on $n$-tuples of $\Omega$. If $G$ is a closed permutation group
on $\Omega$, then $G$ is the full automorphism group
of this structure.

For those actions that are not faithful,
  we say that $G$ is closed in the permutation
topology if the image of $G$ in $\Sym(\Omega)$ under the natural
  homomorphism is closed.  This condition is
equivalent to the condition that
stabilizers of points are closed subgroups of $G$.

Compactness (note that in this work the term {\em compact} does not
include the Hausdorff condition) has a natural
interpretation in the permutation topology.
A subset of a topological space is said to be {\em
 relatively compact} if it has compact closure.
The following lemma slightly generalizes a result by Woess
for automorphism groups of locally finite connected
graphs, and the proof is the same as Woess' proof.

\begin{lemma}{\rm (\cite[Lemma 1 and Lemma 2]{Woess1992})}\label{LCompact}
Let  $G$ be a group acting transitively on a set $\Omega$ and endow
$G$ with the permutation topology.
Assume that  $G$ is closed in the permutation topology and that all suborbits are finite.

(i) The stabilizer $G_\alpha$ of a point $\alpha \in \Omega$ is
compact.

(ii) A subset $A$ of $G$ is
 relatively compact
 in $G$ if and only if the set $A\alpha$ is finite for every $\alpha$
in $\Omega$.

Furthermore, if $A$ is a subset of $G$ and $A\alpha$
is finite for some $\alpha\in\Omega$ then $A\alpha$ is finite for every
$\alpha$ in $\Omega$.
\end{lemma}

\begin{proof}  (i)  Let $K$ denote the kernel of the action of $G$ on
  $\Omega$.  Every open neighbourhood of the identity will contain $K$
  and thus $K$ is compact.  Because $G_\alpha$ is closed we see that
  $K\cap G_\alpha$ is compact.  In order to show that $G_\alpha$
  is compact it is thus enough to show that $G_\alpha/(K\cap
  G_\alpha)$ is compact.  The group $G_\alpha/(K\cap
  G_\alpha)$ acts faithfully on $\Omega$ and the quotient topology on 
$G_\alpha/(K\cap G_\alpha)$ is the same as the permutation topology
  induced by the action on $\Omega$. We may thus assume that
  $G_\alpha$ acts faithfully on $\Omega$.

Let $(\Omega_i)_{i\in I}$ denote the family of $G_\alpha$-orbits on
$\Omega$.  Let $H_i$ denote the permutation group induced by
$G_\alpha$ on $\Omega_i$.  Since, each $\Omega_i$ is finite, then each
group $H_i$ is finite and discrete in permutation topology.  
Set $H=\prod_{i\in I} H_i$, the full Cartesian product.  Note that
$\Omega$ is the disjoint union of the $\Omega_i$'s and that $H$ has a
natural action on $\Omega$.  The
permutation topology on $H$ induced by the action on $\Omega$ is the
same as the product topology.  Thinking of $G_\alpha$ and $H$ as
subgroups of $\sym(\Omega)$ we see that $G_\alpha$ is a closed
subgroup of the compact subgroup $H$ and is thus compact.

(ii)  Suppose that ${A}^-$, the closure of $A$ in $G$,
 is compact.  Then for a point
 $\alpha\in \Omega$ there is a finite open
 covering of ${A}^-$ by sets of the type $gG_\alpha$; that is to
 say,
 we can find $g_1,\ldots,g_n\in G$ such that
 ${A}^-\subseteq\bigcup_{i=1}^n g_iG_\alpha$.  Then $A\alpha\subseteq
 \{g_1\alpha, \ldots,g_n\alpha\}$.

Conversely, suppose that $A\alpha=\{\alpha_1,\ldots,\alpha_n\}$.
 Let $g_i$ be an element in $A$
 such that $g_i\alpha=\alpha_i$.  Then
$A\subseteq\bigcup_{i=1}^n g_iG_\alpha$.
The latter set is compact, so the closure of $A$ is compact.
\end{proof}

\bigskip

Lemma \ref{LCompact} 
implies that if $G$ is a closed transitive permutation group
on a countable set $\Omega$ such that all
suborbits are finite then $G$, with the permutation topology, is a
locally compact, totally
disconnected group.   In particular, the automorphism group of a locally finite,
transitive graph is a locally compact, totally disconnected group.

\bigskip

A subgroup $H$ in a topological group $G$ is said to be cocompact if
$G/H$ is a compact space.  This concept has also a natural
interpretation in terms of the permutation topology.

\begin{lemma}\label{Lcocompact}
{\rm (\cite[Proposition~1]{Nebbia2000}, cf.~\cite[Lemma~7.5]{Moller2002})}
Let $G$ be a group acting transitively on a set $\Omega$. Assume
that $G$ is closed in the permutation topology and all suborbits are
finite.
Then a subgroup $H$ of $G$ is
cocompact if and only if $H$ has finitely many orbits on
$\Omega$.
\end{lemma}


\begin{proof}  Suppose first that $H$ is cocompact.  This means that
both the spaces of right and left cosets of $H$ in $G$ are compact.
Let
$X$ denote the set of right cosets of $H$ in $G$.  The quotient map
$\pi:G\rightarrow X$ is open.  The family of cosets $\{gG_\alpha\}_{g\in
  G}$ is an open covering of $G$ and hence
$\{\pi(gG_\alpha)\}_{g\in G}$ is an open covering of $X$.  Since $X$ is compact, there is a finite subcovering
$\pi(g_1G_\alpha),\ldots, \pi(g_nG_\alpha)$ of $X$.  Then
$G=Hg_1G_\alpha\cup\ldots\cup Hg_nG_\alpha$ and therefore
$\Omega=H(g_1\alpha)\cup\ldots\cup H(g_n\alpha)$.

Conversely, suppose that $H$ has only finitely many orbits on $\Omega$, say
there are elements $g_1, \ldots, g_n$ such that
$\Omega=H(g_1\alpha)\cup\ldots\cup H(g_n\alpha)$.  Then
$G=Hg_1G_\alpha \cup\ldots\cup Hg_nG_\alpha$ and
$X=\pi(g_1G_\alpha)\cup \ldots\cup \pi(g_nG_\alpha)$.  Each of the
sets $\pi(g_1G_\alpha)$ is compact.  Hence $X$, the set of right
cosets of $H$ in $G$, is
compact, because it is a union of finitely many compact
sets.   \end{proof}

\bigskip

  Ideas from
permutation group theory and the permutation topology can  be
applied to the study of a topological group $G$.  For an
open subgroup
$U$  we set $\Omega=G/U$, the space of left
cosets.   The stabilizers in $G$ of points in $\Omega$ are conjugates
of $U$, and thus also open subgroups of $G$.  The stabilizer of a finite set
$\Phi=\{\alpha_1, \ldots, \alpha_n\}$ of points is just the
intersection of the  open subgroups $G_{\alpha_1}, \ldots,
G_{\alpha_n}$ and is thus open in $G$.  From this we see that the
permutation topology coming from the action of $G$ on $\Omega$ is
contained in the topology on $G$.  If the topological group
$G$ is assumed to be
 a totally disconnected and
locally compact, then we can choose  $U$ to
be a compact open subgroup of $G$
(by a theorem of van Dantzig \cite{Dantzig1936}).
In particular, the stabilizers in $G$ of
points in $\Omega$ are all compact (they are conjugates of $U$).
 The second part of  Lemma~\ref{LCompact}
above also holds for the action of $G$ on $\Omega$, so a subset $A$ of
$G$ has compact closure if and only if $A\alpha$ is finite for every
$\alpha \in \Omega$.    This
implies that $|G_\alpha\beta|< \infty$ for all points
$\alpha$ and $\beta$ in $\Omega$.  This is because
$|G_\alpha\beta|=|G_\alpha:G_\alpha\cap G_\beta|$ and this index is
finite because $G_\alpha\cap G_\beta$ is an open subgroup of the
compact group $G_\alpha$.

\subsection{Suborbits and the modular function}
In this section the general assumption will be that $G$ is a closed
permutation group acting on a set $\Omega$ and that
all suborbits of the group $G$ are finite.

Under the above assumptions the group $G$ is a locally compact, totally
disconnected group.  In this section we want to interpret the modular
function on $G$ in terms of the action of $G$ on $\Omega$.
The connection between suborbits and the modular function can be seen from
the following argument due to  Schlichting \cite{Schlichting1979}, see
also \cite{Trofimov1985a}.

Let $\mu$ be a right
Haar-measure on $G$.  Define the modular function $\Delta$
so that if
$A$ is a measurable set then $\mu(gA)=\Delta(g)\mu(A)$.

\begin{lemma} \label{LModular}
{\rm (\cite[Lemma 1]{Schlichting1979}, cf. \cite[Theorem 1]{Trofimov1985a})})
Let $G$ be a closed, transitive permutation
group on a set $\Omega$.  Assume furthermore that all
suborbits of $G$ are finite.  Let $\Delta$ denote the modular function
on $G$.  If  $h$ is an element in $G$ with $h\alpha=\beta$ then
  $$\Delta(h)=\frac{|G_\beta\alpha|}{|G_\alpha\beta|}=
\frac{|G_\alpha:G_\alpha\cap h^{-1}G_\alpha h|}
{|G_\alpha:G_\alpha\cap hG_\alpha h^{-1}|}.$$
\end{lemma}

\noindent
\begin{proof}
  Then, with $\mu$ denoting the right
Haar-measure on $G$,
\begin{align*}
    \mid G_\beta\alpha\mid & =  \mid G_\beta:G_\alpha\cap G_\beta\mid \\
                           & =  \mu(G_\beta)/\mu(G_\alpha\cap G_\beta) \\
                           & =  \mu(hG_\alpha h^{-1})/
                                              \mu(G_\alpha\cap G_\beta) \\
                           & =  \Delta(h)\mu(G_\alpha )/
                                              \mu(G_\alpha\cap G_\beta) \\
                           & =  \Delta(h)
                              \mid G_\alpha:G_\alpha\cap G_\beta\mid \\
                           & =  \Delta(h)\mid  G_\alpha \beta\mid.
\end{align*}
And we see that $\Delta(h)={|G_\beta \alpha|}/{|G_\alpha\beta|}$.
\end{proof}

\bigskip

{\em Remark.}  Let $G$ be a locally compact, Hausdorff group with
modular function $\Delta$.  Assume $G$ acts transitvely on a set $\Omega$
such that the stabilizers of points are compact open subgroups of $G$.
The calculation in the proof of Lemma~\ref{LModular} is also valid in
this case and thus the conclusion in  Lemma~\ref{LModular} holds also.

\smallskip

For an orbital $A=G(\alpha, \beta)$ we define the {\em paired orbital}
as the orbital $A^*=G(\beta, \alpha)$.  This pairing of orbitals also
gives us a pairing of suborbits, where the
suborbit $G_\alpha\beta$ is paired to a suborbit $G_\alpha\gamma$
where $\gamma$ is a point in $\Omega$ such that $(\alpha,\gamma)$ is
in $G(\beta, \alpha)$.  Denote by $\Gamma$ the directed orbital graph for the
orbital $G(\alpha, \beta)$.  The size of the suborbit $G_\alpha\beta$
is the out-valency of $\Gamma$ and the size of the paired suborbit
$G_\alpha\gamma$ is the in-valency of $\Gamma$.
Using Lemma~\ref{LModular}, we obtain the following proposition.

\begin{proposition}{\rm (\cite[Theorem 1]{Trofimov1985a})}
Let $G$ be a closed, transitive permutation
group on a set $\Omega$.  Assume that
all suborbits of $G$ are finite.  Then the lengths of paired suborbits
are always equal if and only if ${G}$ is unimodular.
\end{proposition}

Using that the modular function is a homomorphism, we obtain the following.

\begin{corollary}
Let $G$ be a closed, transitive permutation
group on a set $\Omega$.
Let $K=\ker(\Delta)$.  Then

(a) If $\alpha\in \Omega$ then $G_\alpha\leq K$.

(b) Let $G'$ be the derived group of $G$.  Then $G'\leq K$.

(c) If $g$ is an torsion element of $G$ then $g\in K$.
\end{corollary}

\noindent
\begin{proof} Follows directly from the definitions and the fact
that $G/K$
is isomorphic to a multiplicative subgroup of the positive real numbers. \end{proof}

\begin{corollary}
Let $G$ be a closed, transitive permutation
group on a set $\Omega$.  Assume that
all suborbits of $G$ are finite.
If $G$  is also
primitive, then paired suborbits have equal length.
\end{corollary}

\noindent
\begin{proof}  Let $K=\ker(\Delta)$.  Then, because $G$ is primitive
and $K$ is normal in $G$,
we know that either $K=\{e\}$ or $K$ is transitive.  If $K=\{e\}$
then $G$ is abelian and must in fact be trivial.
If $K$ is transitive
then the result of Lemma~\ref{LModular} implies that lengths of paired
suborbits are equal.  \end{proof}

\bigskip

For a long time it was an open question, posed by Peter M.~Neumann,
whether  one could have a group
acting primitively on a set $\Omega$ with a finite suborbit
paired to an infinite one.   Such examples were constructed by David M.~Evans in \cite{Evans2001}.
\bigskip

Consider a connected graph $\Gamma$ and assume that $G$ is a closed
subgroup of $\autga$ that acts transitively on $\Gamma$ and that all
suborbits are finite.  For convenience we
think of each undirected
edge $\{\alpha,\beta\}$
in $\Gamma$ as consisting of two directed edges $(\alpha,\beta)$ and
$(\beta,\alpha)$.  Each directed edge $e=(\alpha,\beta)$
will be labeled by a number
$$\Delta_e=\frac{|G_\beta \alpha|}{|G_\alpha \beta|}.$$
Observe that $\Delta_{(\alpha,\beta)}=\Delta_{(\beta,\alpha)}^{-1}$.
Furthermore, note that if $g$ is an element of $G$
such that $g\alpha=\beta$ and $e=(\alpha, \beta)$ then
$\Delta(g)=\Delta_e$.  Suppose $g$ is an element of $G$
and $g\alpha=\gamma$ and
that there is a vertex
$\beta$ in $X$ such that $(\alpha,\beta)$  and $(\beta,\gamma)$ are
edges in $X$.  Find elements $g_1$ and $g_2$ in $G$ such that
$g_1\alpha=\beta$ and $g_2\beta=\gamma$.
Then $g\alpha=g_2g_1\beta$ and from the
formula above for the modular function we can deduce that
$\Delta(g)=\Delta(g_2g_1)$.  We also find that
$$\Delta(g)=\Delta(g_2g_1)=\Delta(g_2)\Delta(g_1)=\Delta_{(\beta,\gamma)}
\Delta_{(\alpha,\beta)}.$$
This can be extended to directed walks
of arbitrary length, so that if $g\alpha=\beta$
then we take a directed walk from $\alpha$ to $\beta$, enumerate the edges in
the walk as $e_1, \ldots, e_k$ and then
$$\Delta(g)=\Delta_{e_1}\cdots \Delta_{e_k}.$$
Hence the labeled graph completely describes the modular function on
$G$. This idea can be found in the paper \cite{BassKulkarni1990} by Bass and Kulkarni.

\bigskip
Suppose now not only that all the suborbits of $G$ are finite
but also that there is a finite upper
bound $m$ on their length.
In that case, Lemma \ref{LModular} implies that the image of the modular
function is a bounded set.  The image of the modular function is a bounded subgroup
of the multiplicative group of positive
real numbers and is thus the trivial subgroup.  Hence $G$ is unimodular.
We have deduced the following unpublished result of Praeger.

\begin{corollary}
Let $G$ be a closed, transitive permutation
group on a set $\Omega$ such that
all suborbits of $G$ are finite.
If there is a finite bound on the length of suborbits, then paired suborbits
have equal length.
\end{corollary}

The next result, also due to Praeger \cite{Praeger1991},
uses the modular function to infer information about graph structure.
The {\em directed integer graph}
$\Z$ has  the set of integers as a vertex set and the edge set is the set of
all ordered pairs $(n, n+1)$.

\begin{theorem} {\rm (\cite{Praeger1991})} \label{Tepimorphism}
Let $\Gamma$ be an infinite, connected, vertex and edge transitive directed
graph with finite but unequal in- and out-valence.  Then there is a
graph epimorphism $\varphi$ from $\Gamma$ to the directed integer graph
$\Z$.  For each $i\in \Z$, the inverse image $\varphi^{-1}(i)$ is
infinite.
\end{theorem}

\begin{proof}  Write $G=\autga$.  Let $q=d^-/d^+$ where $d^+$ is the
out-valence of $\Gamma$ and $d^-$ is the in-valence of $\Gamma$.
Consider an edge $(\alpha,\beta)$ in $\Gamma$.  Because $G$ acts
transitively on the edges of $\Gamma$ we can conclude that
$|G_\alpha \beta|=d^+$
and $|G_\beta \alpha|=d^-$ and hence if $g\alpha=\beta$
then $\Delta(g)=d^-/d^+$.  Using the graph $\Gamma$ to calculate the
modular function in a similar way as described above we
conclude that for every element $g\in G$ there is an integer $i$ such
that $\Delta(g)=q^i$.  Hence, if $K$ denotes the kernel of the modular
function then $G/K=\Z$.  Fix a vertex $\alpha_0$ in $\Gamma$ and define a map
$\varphi:V\Gamma\rightarrow \Z$ so that $\varphi(\beta)=i$ if there is an
element $g$ in $G$ such that $g\alpha_0=\beta$ and $\Delta(g)=q^i$.
It is clear that the choice of $g$ is immaterial.  From the way
one uses $\Gamma$ to calculate the modular function  we see that if
$(\alpha,\beta)$ is
an edge in $\Gamma$ then $\varphi(\beta)=\varphi(\alpha)+1$ which implies that
$\varphi$ is a homomorphism from $\Gamma$ to the directed integer
graph.

Assume now, seeking contradiction, that $\varphi^{-1}(i)$ is finite for some $i$.
Note that the fibers of $\varphi$ are just the orbits of the kernel of
the modular homomorphism and are thus  blocks of imprimitivity for
$G$.  Hence, all the fibers of $\varphi$ have the same cardinality,
say $k$.    The number of edges going out of $\varphi^{-1}(0)$ is 
$d^+k$ and the number of edges going into $\varphi^{-1}(1)$ is
$d^-k$.  But, both these numbers should be equal to the number of
edges going from $\varphi^{-1}(0)$ to $\varphi^{-1}(1)$ and because we
are assuming that $d^-\neq d^+$ we have a contradiction.

(A similar proof of Praegers result is in a paper by Evans
\cite{Evans1997}.)\end{proof}

\bigskip
{\em Remark.}  In the next section {\em highly arc transitive}
  digraphs are discussed.  A digraph $\Gamma$ satisfying the conditions
  in Theorem~\ref{Tepimorphism}
  need not be highly arc transitive, but it is easy to show that if
  $d^+$ and $d^-$ are coprime, then $\Gamma$ must be highly arc
  transitive.

\bigskip

The next result we discuss, is a remarkable theorem of
Schlichting \cite{Schlichting1980}.

\begin{theorem}\label{TSchlichting}{\rm (\cite{Schlichting1980})}
Let $G$ be a group acting transitively on a set $\Omega$.  Then
there is a finite bound on the sizes of suborbits of $G$ if and only
if there is a $G$-invariant equivalence relation $\sim$ on $\Omega$ with finite
classes, such that the stabilizers of points in the action of $G$ on
$\Omega/\sim$ are finite.
\end{theorem}

While this is a theorem about
permutation groups, Schlichting's proof utilizes various concepts
from functional analysis and a theorem of Iwasawa
\cite[Theorem~1]{Iwasawa1951}.
Later Theorem~\ref{TSchlichting} was rediscovered by
Bergman and Lenstra \cite{BergmanLenstra1989}, who gave a group
theoretical/combinatorial proof.

The theorem of Iwasawa that Schlichting uses in his proof of 
Theorem~\ref{TSchlichting} is about the
relationship between the classes [{\em IN}] and   [{\em SIN}] of
    topological groups.

\begin{definition}
A locally compact group is said to be in the class [{\em IN}] if there is
  a compact neighbourhood $K$ of the identity (i.e.~$K$ contains an open
  set containing the identity) that is invariant under
  conjugation by elements in $G$.

A locally compact group is said to  be in the class [{\em SIN}] if every
  neighbourhood of the identity contains
  a compact neighbourhood $K$ of the identity that is invariant under
  conjugation by elements in $G$.
\end{definition}

Let us start by relating these two properties to the permutation
topology.

\begin{proposition}\label{PIN}
Let $G$ be a transitive permutation group on a set $\Omega$ and
assume that all suborbits are finite.  Then $G$ with the permutation
topology is in the class [{\em IN}] if and only if there is a finite bound
  on the sizes of suborbits.
\end{proposition}

\begin{proof}  First assume that $G$ is in the class [{\em IN}].
Suppose  $\alpha,\beta\in \Omega$.  We want to find a constant upper bound,
  independent of
 $\alpha$ and $\beta$, for the size of the suborbit $G_\alpha \beta$.
  Let $K$ be a compact neigh\-bourhood of the identity that is invariant
  under conjugation.  Since $K$ contains an open neigh\-bourhood of the
  identity we can find a finite set $\Phi$ such that $G_{(\Phi)}\subseteq
  K$.
Choose a point $\gamma$ in $\Omega$.
Because $K$ is compact,
  $|K\gamma|=m<\infty$.  Let $k$ be an upper bound for the indices
$|G_\delta:G_{(\Phi)}|\leq k$ with $\delta\in \Phi$.
  Find an element
  $f\in G$ such that $f\beta=\gamma$ and set $\alpha'=f\alpha$.  Whence
$|G_\alpha \beta|=|G_{\alpha'} \gamma|$.  Then we find an element $h$
  such that $\alpha'\in
  h\Phi$.  Note that $G_{(h\Phi)}=hG_{(\Phi)}h^{-1}$.  Since $K$ is invariant
  under conjugation and $G_{(\Phi)}\subseteq K$,
we conclude that  $G_{(h\Phi)}\subseteq K$.  There\-fore
  $|G_{(h\Phi)}\alpha|\leq
  |K\alpha|=m$.  We also know that $|G_{\alpha'}:G_{(h\Phi)}|\leq k$ and thus
$|G_{\alpha} \beta|=|G_{\alpha'} \gamma|\leq k|G_{(h\Phi)}\gamma|\leq km$.  Hence  $km$ 
is an upper bound for the size of suborbits of $G$.

Conversely, assume that there is an upper bound $m$ on the sizes of
suborbits.  For a finite subset $\Phi$ of $\Omega$ we let $m(\Phi)$ denote the
size of the largest orbit of $G_{(\Phi)}$.  We choose  a finite subset
$\Phi$ such that $m_0=m(\Phi)$ is as small as possible.  Define
$A=\bigcup_{g\in G} gG_{(\Phi)}g^{-1}$.  Clearly, the set $A$ is open
and invariant under conjugation by elements of $G$.
We claim that $A$ is relatively
compact.  By Lemma~\ref{LCompact}(ii), 
we only need to show that $A\alpha$ is finite for some
$\alpha\in \Omega$.  Choose $\alpha\in \Omega$ such that
$|G_{(\Phi)}\alpha|=m(\Phi)$.  For this $\alpha$ we will show that
$|A\alpha|\leq m$.  Arguing by contradiction, suppose
$|A\alpha|\geq n>m(\Phi)$.  Take $f_1, \ldots, f_n\in A$ such that the
elements $f_1\alpha, \ldots, f_n\alpha$ are all distinct.   Since $f_1, \ldots,
f_n\in A=\bigcup_{g\in G} gG_{(\Phi)}g^{-1}$, we can find elements $g_1,
\ldots, g_n\in G$ such that $f_i\in g_iG_{(\Phi)}g_i^{-1}=G_{(g_i\Phi)}$.
Write $\Phi_i=g_i\Phi$ and set
$E=\{f_1\alpha, \ldots, f_n\alpha\}\cup \Phi_1\cup\cdots\cup
\Phi_n$.  Then $m(E)=m_0$.  Let $\Delta$ denote a $G_{(E)}$-orbit of size $m_0$.  Note that $\Delta$ is also a $G_{(\Phi_i)}$-orbit.
Choose an element $\delta\in \Delta$.  There is for each
$i=1, \ldots, n$ an element $h_i\in G_{(E)}$ such that
$h_i\delta=f_i\delta$ and therefore
$h_i^{-1}f_i\in G_\delta$.  But
$h_i^{-1}f_i\alpha=f_i\alpha$ and we can conclude that $|G_\delta
\alpha|\geq n>m$
contrary to assumptions.  (The above argument is related to the proof
of Theorem~\ref{TSchlichting} by Bergmann and Lenstra
\cite{BergmanLenstra1989}, but this version is from a lecture course
given by Peter M.~Neumann  in Oxford 1988--1989.)  \end{proof}

\bigskip

\begin{proposition}\label{PSIN}
Let $G$ be a transitive permutation group on a set $\Omega$ and
assume that all suborbits are finite.  Then $G$ with the permutation
topology is in the class [{SIN}] if and only $G$ is discrete
(i.e. the stabilizers of points are finite).
\end{proposition}

\begin{proof}  It is obvious that if $G$ is discrete then $G$ is
in [{\em SIN}].

Let us now assume that $G$ is in [{\em SIN}].  Then, if $\alpha$
denotes a point in $\Omega$,  the open subgroup $G_\alpha$ contains a compact
neighbourhood $K$ of the identity that is invariant under conjugation.
Then for $g\in G$ we see that $K=gKg^{-1}\subseteq gG_\alpha
g^{-1}=G_{g\alpha}$ and
thus $K\subseteq \bigcap_{g\in G}gG_\alpha g^{-1}=\bigcap_{g\in G}
G_{g\alpha}$.
Because $G$ is
assumed to be transitive we conclude that $K$ fixes every point of
$\Omega$.  But we are also assuming that $G$ acts faithfully on $\Omega$ so
$K=\{e\}$
and $G$ is discrete.  \end{proof}

\bigskip

The theorem of Iwasawa mentioned above says that if a locally compact
group is in the class [{\em IN}] then there is a compact normal subgroup $N$
such that $G/N$ is in the class [{\em SIN}].

\bigskip

\begin{proof}  (Theorem~\ref{TSchlichting})  Assume that there is a
finite upper bound on the sizes of suborbits of $G$.
By Proposition~\ref{PIN} the
group $G$ is in the class [{\em IN}].  Iwasawa's Theorem gives us a
compact, normal subgroup $N$ of $G$ such that $G/N$ is in  [{\em
    SIN}].  The orbits of the normal subgroup $N$ are the classes of a
$G$-invariant equivalence relation $\sim$ on $\Omega$
and because $N$ is compact
these classes are all finite. The group $H=G/N$ certainly
is in [{\em SIN}], but it is not certain that $H$ acts faithfully on
$\Omega'=\Omega/\sim$ so we can not apply Proposition~\ref{PSIN} directly.
Note that if $\alpha\in\Omega'$ then $H_\alpha$ is an open subgroup of
$H$.  Since $H$ is in  [{\em SIN}] there is a compact invariant
neighbourhood neighbourhood $K$ contained in $H_\alpha$.  As in the
proof of Proposition~\ref{PSIN} we conclude that $K$ is contained in
the kernel of the action of $H$ on $\Omega'$.  Thus the kernel $N'$ of the
action of $H$ on $\Omega'$ is an open subgroup of $H$.  The group
$H/N'$ can be regarded as a permutation group on $\Omega'$.  From this we
conclude that the permutation topology on $H/N'$ is discrete, which
implies that the stabilizer in $H/N'$ of a point $\alpha\in \Omega'$
must be finite.

The proof of the other direction is left to the reader.  \end{proof}

\bigskip
Schlichting's Theorem implies the following general result about
totally disconnected, locally compact groups.

\begin{corollary}  Let $G$ be a totally disconnected locally compact group.

(i)  The group $G$ has a compact open normal subgroup if and only if
  there is a compact open subgroup $U$ and a number $m_U$ such that
$|U:U\cap gUg^{-1}|\leq m_U$ for all $g\in G$.

(ii) If there is such a number $m_U$ for one compact open subgroup
  $U$ then there is a number $m_V$ for any compact open subgroup
  $V$ such that  $|V:V\cap gVg^{-1}|\leq m_V$ for all $g\in G$.
\end{corollary}

\begin{proof}  (i) If $G$ contains a compact open normal subgroup $N$, then
 we can take $U=N$ and
$m_N=1$.

Conversely, assume that $U$ is a compact open subgroup and there is a number
$m_U$ such that $|U:U\cap gUg^{-1}|\leq m_U$ for all $g \in G$.
Put $\Omega=G/U$.  Take a point $\alpha$ in $\Omega$
such that $U=G_\alpha$.  Note that if $g\in G$ and $\beta=g\alpha$
then $|G_\alpha\beta|=|G_\alpha:G_\alpha\cap G_\beta|=|U:U\cap
gUg^{-1}|\leq m_U$.  Thus there is a finite upper bound on the sizes of
suborbits and we can apply Schlichting's Theorem, which
    provides  us with a
$G$-invariant equivalence relation $\sim$ on $\Omega$ with finite classes
such that the stabilizer of a $\sim$-class acts like a finite group
on $\Omega'=\Omega/\sim$.  This in turn implies that $N$, the kernel
of the action of $G$ on $\Omega'$, is a open normal subgroup of $G$.
If $\alpha\in \Omega$ then $N\alpha$ is contained in the $\sim$-class
of $\alpha$ and thus $N\alpha$ is finite.  Therefore $N$ is also
compact.

(ii)  From the proof of statement (i) we get the existence of a compact open normal
subgroup $N$.  Consider the action of $G$ on the set $\Omega=G/V$.  The
orbits of $N$ on $\Omega$ are all finite and give us the equivalence
classes of a $G$-invariant equivalence relation $\sim$.  Let $k$ be the
number of element in a $\sim$-class.  The group $G$
acts on $\Omega'=\Omega/\sim$ and the stabilizers of points in
$\Omega'$ act like finite groups on $\Omega'$.  Thus the sizes of
suborbits in the action of $G$ on $\Omega'$ are bounded above by some
number $l$.  Let $\tilde{\alpha}$
denote the $\sim$-class of an element $\alpha\in\Omega$.  For
elements $\alpha$ and $\beta$ in $\Omega$ we see that
$|G_\alpha\beta|\leq |G_{\tilde{\alpha}}\beta|\leq
k|G_{\tilde{\alpha}}\tilde{\beta}|=kl$.   From this it follows that
 $|V:V\cap gVg^{-1}|\leq kl$.  \end{proof}

\subsection{The theorems of Trofimov}\label{STrofimov}
As already mentioned, the automorphism group of a locally finite, connected
graph with the permutation topology is locally compact and totally
disconnected.  In this section we will discuss three theorems of
Trofimov (see \cite{Trofimov1984},
\cite{Trofimov1985}, \cite{Trofimov1987}).
The conclusions in them all resemble the conclusion in Schlichting's
Theorem, but none of the theorems is proved by referring to
Schlichting's Theorem.
We start with the earliest of these three theorems.  First, we
explain the terminology used.

An automorphism $g$ of a connected graph
$\Gamma$ 
is said to be {\em bounded} if there is a constant $c$ such that
$d_\Gamma(\alpha,g\alpha)\leq c$ for all vertices $\alpha$ in
$\Gamma$.

\begin{theorem}{\rm (\cite{Trofimov1984})}\label{TTrofimovBounded}
Let $\Gamma$ be a locally finite, transitive graph and $B(\Gamma)$ be the
subgroup of
bounded automorphisms.  The following assertions are equivalent:

(i) The subgroup $B(\Gamma)$ is transitive.

(ii) There is an equivalence relation $\sim$ with finite equivalence
classes on the vertex set of $\Gamma$
such that $B(\Gamma)$ acts on $\Gamma$ like a finitely generated free abelian
group.
\end{theorem}

In a connected graph $\Gamma$  we define {\em the ball of
radius} $n$ {\em with
center} in a vertex $\alpha$ as the set $B_n(\alpha)=\{\beta\in
  V\Gamma\mid d(\alpha, \beta)\leq n\}$.
The bounded
automorphisms of $\Gamma$ are related to topological properties of $\autga$ via the
following result of Woess.   With a later application in mind,
Woess' result is stated for metric spaces rather than
just graphs.  A (closed) {\em ball of radius $r$ with center }
$\alpha$ in a metric space $X$ 
as the set  $B_r(\alpha)=\{\beta\in
  X\mid d_X(\alpha, \beta)\leq r\}$.
An isometry $g$ of a metric space $X$ 
is said to be {\em bounded} if there is a constant $c$ such that
$d_X(\alpha,g\alpha)\leq c$ for all points $\alpha$ in
$X$.
Recall also that an element $g$ in a
topological group $G$ is called an FC\/$^-$-element if the conjugacy
class of $g$ has compact closure.

\begin{lemma}
{\rm (Cf. \cite[Lemma 4]{Woess1992})}\label{LWoessBounded}
Suppose $G$ is a topological group acting transitively
by isometries on a metric space $X$.  Assume furthermore that the
stabilizer in $G$ of a point in $X$ is a compact open subgroup and
that for every value of $n$ the ball
$B_n(\alpha)$ is finite.
An element $g\in G$  is bounded if and only if $g$ is an
FC\/$^-$-element of $G$.
\end{lemma}

\begin{proof}  Suppose $g\in G$ acts as a bounded isometry on $X$.
Find a number $M$ such that $d(g\alpha,\alpha)\leq M$ for every
$\alpha\in X$.
For $h\in G$, write $g^h=hgh^{-1}$.  Set $g^G=\{g^h\mid h\in G\}$.
It is clear that
$d(g^h\alpha ,\alpha)=d(gh^{-1}\alpha , h^{-1}\alpha)\leq M$
for every $\alpha\in X$.
   We see that the set
$g^G\alpha$ is finite and by Lemma~\ref{LCompact}(ii)
the conjugacy class $g^G$ has compact
closure.

Conversely, suppose that the conjugacy class
$g^G$ has compact closure.  Then, for every $\alpha\in X$
the set $g^G\alpha$ is finite.  Take a number $M$
such that $d(g^h\alpha,\alpha)\leq M$
for every $h\in G$.  Take some $\beta\in X$.
Choose $h\in G$ so that $\beta =h^{-1}\alpha$.  Then
$d(g\beta,\beta)=d(gh^{-1}\alpha, h^{-1}\alpha)=
d(g^h\alpha,\alpha)\leq M$.  So $g$ acts on $\Gamma$ as a
bounded automorphism.    \end{proof}

\medskip

This connection with topological notions
can be used to give a short proof of Theorem~\ref{TTrofimovBounded},
where only elementary results from the
theory of topological groups are used, see \cite{Moller1998}.

\bigskip

A locally finite graph $\Gamma$ is said to have {\em polynomial
  growth} if the number of vertices of $\Gamma$ in 
$B_n(\alpha)$ is bounded above by a polynomial in $n$.
It is easy to see that this property does not depend on the choice of
the vertex $\alpha$.
A finitely generated group $G$ is said to have polynomial growth if
its Cayley graph with respect to a finite generating set
has polynomial growth (the choice of
generating sets is immaterial, since having polynomial growth is a
quasi-isometry invariant).

The second theorem of Trofimov related to Schlichting's Theorem
is the following:

\begin{theorem} \label{TTrofimovPoly}
{\rm (\cite[Theorem~2]{Trofimov1985})}
Suppose $\Gamma$ is a connected, locally finite graph with polynomial
growth, and $G$ is a group that acts transitively on $\Gamma$.
Then there is a $G$-invariant equivalence relation
$\sim$ with finite classes on the vertex set of $\Gamma$
such that the quotient of $G$ by the kernel of the induced
action on $\Gamma/\sim$ is a finitely generated, virtually
nilpotent group with finite stabilizers for vertices of
 $\Gamma/\sim$.

\end{theorem}

It should be noted that Trofimov proves an even stronger result
\cite[Theorem~1]{Trofimov1985}, since
he shows that it is possible to find an equivalence relation
$\sim$ as described in Theorem~\ref{TTrofimovPoly} such that
the stabilizer of a vertex in $\aut(\Gamma/\sim)$ is
finite.

The theorem of Trofimov can be seen as a graph theoretical version of
Gromov's celebrated theorem characterizing finitely generated groups
with polynomial growth, see \cite{Gromov1981}.  Indeed, Trofimov
uses Gromov's Theorem in his proof.
A version of
Gromov's theorem for topological groups has been proved by Losert
in \cite{Losert1987}.
Woess in \cite{Woess1992} used Losert's version of Gromov's Theorem
from \cite{Losert1987}
to give a short proof of Theorem~\ref{TTrofimovPoly}.

We will be returning to
polynomial growth and Trofimov's result in Section~\ref{SPolynomial}.

\bigskip

There is a third theorem of Trofimov's with a similar feel to it as the
two theorems stated above.  This theorem involves the concept of
an $o$-automorphisms of a graph.  An automorphism $g$ of a connected graph
$\Gamma$ is called an {\em $o$-automorphism} if
$$\max\{d(\beta, g\beta)\mid \beta\in V\Gamma, d(\alpha,\beta)\leq n\}=o(n),$$
where $\alpha$ is a fixed vertex.  It is easy to show that this property
does not depend on the choice of the vertex $\alpha$.  It is also easy to
prove that the $o$-automorphisms form a normal subgroup of $\autga$.

\begin{theorem} \label{TTrofimovoauto}
{\rm (\cite[Corollary~1]{Trofimov1987})}
Suppose $\Gamma$ is a connected, locally finite graph
and $G$ is a group that acts transitively on $\Gamma$.  Then
the following are equivalent:

(i) $G\leq o(\autga)$

(ii)  There is a $G$-invariant equivalence relation $\sim$ on the vertex
set of $\Gamma$ such that the equivalence classes of $\sim$ are
finite and
if $K$ denotes the kernel of the action of $G$ on the equivalence
classes then $G/K$ is a finitely generated nilpotent group acting
regularly on $\Gamma/\sim$.
\end{theorem}

Trofimov's proofs of these three theorems are long and
difficult.  The proofs mentioned above of the first two theorems,
are
short, but admittedly, in the proof of  Theorem~\ref{TTrofimovPoly}
the results from the theory of topological groups
used are highly non-trivial.  It would be interesting to find
a topological interpretation of the concept
of an $o$-automorphism.   Possibly that could lead to a shorter proof of
Theorem~\ref{TTrofimovoauto}.

\section{The scale function and tidy subgroups}\label{STidyScale}

\indent

The theory of locally compact groups is the part of the theory of
topological groups that has widest appeal and most applications in
other branches of mathematics.  When
looking at locally compact groups there are the connected
groups on one end of the spectrum
and the totally disconnected groups on the other
end.

The fundamental result in the theory of locally compact totally
disconnected groups is the theorem of van Dantzig \cite{Dantzig1936} that
such a group must always contain a compact open subgroup.
A big step towards a general theory was taken in
the paper \cite{Willis1994} by Willis.
The fundamental concepts of Willis's theory are
the scale function and tidy subgroups.

\begin{definition}\label{DTidyscale}
Let $G$ be a locally compact totally
disconnected group and $x$ an
element in $G$. For a compact open subgroup $U$ in $G$ define
$$U_+=\bigcap_{i=0}^\infty x^iUx^{-i}\qquad\mbox{and}\qquad
      U_-=\bigcap_{i=0}^\infty x^{-i}Ux^{i}.$$
Say $U$ is {\em tidy} for $g$ if\\
{(TA)}  $U=U_+U_-=U_-U_+$\\
and \\
{(TB)} $U_{++}=\bigcup_{i=0}^\infty x^iU_+x^{-i}$ and
$U_{--}=\bigcup_{i=0}^\infty x^{-i}U_- x^{i}$ are both closed in $G$.
\par
  Let $G$ be a locally compact totally
disconnected group.
The {\em scale function} on $G$ is defined as
 $$\s(x)=\min\{|U:U\cap x^{-1}Ux|: U \mbox{\rm\ a compact open subgroup of }
G\}.$$
\end{definition}

The connection between the scale function and tidy subgroups is
described in the following theorem due to Willis.

\begin{theorem}
\label{TScaletidy}
{\rm (\cite[Theorem 3.1]{Willis2001a})}
Let $G$ be a totally disconnected, locally compact group and $g\in G$.
Then $\s(g)=|U:U\cap g^{-1}Ug|$ if and only if $U$ is tidy for $g$.
\end{theorem}

{\em Remark.}   Instead of stating our results for totally disconnected, locally compact groups, we could phrase our results for  locally compact groups, that contain a compact, open subgroup.

\medskip

Now on to something completely different.

\subsection{Highly arc transitive digraphs}\label{Shat}
\begin{definition}
  A digraph $\Gamma$ is called $s$-{\em arc transitive} if the
  automorphism group acts transitively on the set of  $s$-arcs.   If $\Gamma$ is $s$-arc transitive for all numbers $s\geq 0$
  then $\Gamma$ is said to be {\em highly arc transitive}.
\end{definition}

We also say that a group $G\leq \autga$ acts {\em highly arc transitively}
  on $\Gamma$ if $G$ acts transitively on the $s$-arcs in $\Gamma$ for
  all $s$.

  The definition of highly arc transitive digraphs occurs first in the
  paper  \cite{CPW1993} by Cameron, Praeger and
  Wormald.  Similar conditions, both for directed and undirected
  graphs,  have been studied by various authors in various contexts.

Let us start by looking at several examples.

\bigskip

{\em Examples.}  (i)  Let $\Gamma$ be a directed tree with constant
in- and out-valencies.  Clearly $\Gamma$ is highly arc transitive.

(ii)  Let $\Gamma$ be a digraph with the set $\Q$ of rational numbers
as a vertex set and $(\alpha, \beta)$ an edge in $\Gamma$ if and only
if $\alpha >\beta$.  Again it is clear that $\Gamma$ is a highly arc
transitive digraph.

(iii)  (Cf.~\cite[Example 1]{Moller2002})
Let $T_1$ denote the regular directed tree in which every
vertex has in-valency 1 and out-valency $q$.  Let $L=\ldots,
\alpha_{-1},\alpha_0,\alpha_1,\alpha_2,\ldots$ be a directed line in $T_1$.
Define
$$H=\{h\in\aut(T_1)\mid \mbox{there is a number }i \mbox{ such that }
h\alpha_i=\alpha_i\}.$$  If $h\in H$ and $h$ fixes some $\alpha_i$ then $h$
also fixes all vertices $\alpha_j$ with $j<i$.  One can also see that the
orbits of $H$ are infinite and each orbit contains precisely one vertex
from $L$.  The orbits of $H$ are called {\em horocycles}.  The
horocycles could also be defined without reference to
the automorphism group.  Then we could define two vertices $\alpha$ and
$\beta$ to be in the same {\em horocycle} if there is a number $n$
such that the unique path in $\Gamma$ from $\alpha$ to $\beta$
starts by going backwards along
$n$ arcs and then going forward along $n$ arcs.
Let $C_i$
denote the horocycle containing $\alpha_i$.

For each $i\in\Z$ take $r-1$ copies $S^1_i, \ldots, S^{r-1}_i$ of $T_1$
and let $\psi^j_i:S^j_i\rightarrow T_1$ be an isomorphism.   The preimage
of the horocycle $C_i$ is a horocycle $B^j_i$ in $S^j_i$.  When restricted,
$\psi^j_i$ defines an isomorphism between the digraphs spanned by
$\des(B^j_i)$ and $\des(C_i)$.  Use
this partial isomorphism to identify the vertices in $\des(B^j_i)$ with the vertices in
 $\des(C_i)$.
 Do this for every $i$ and you get a new digraph $T_2$.  The
 digraph $T_2$ is far from being a tree, but if  $\alpha$ is a vertex in
 $T_2$ then the digraph spanned by
 $\des(\alpha)$ is a rooted infinite directed $q$-ary tree.  The
 vertices in $T_2$ that did belong to $T_1$ now all have out-valency
 equal to $q$,
 and in-valency equal to $r$.
 Look at the part of $S^j_i$ that did not get identified with vertices
 in $T_1$.  This part is a union of horocycles, at each horocycle in it we
 glue $r-1$ new copies of $T_1$ in the same fashion.  Do this for each $i$   and
 each horocycle in $S^j_i$, not belonging to $T_1$, and
 get a digraph $T_3$.  Continuing in the same fashion we construct a
 sequence $T_1\subseteq T_2\subseteq T_3\subseteq\ldots$ of digraphs.  In
 the end we get a digraph $DL(q,r)=\bigcup T_i$.  In this digraph every
 vertex has
 in-valency equal to $r$ and out-valency equal to $q$.
If $\alpha$ is a vertex in
$DL(q,r)$, then the subdigraph spanned by $\des(\alpha)$
 is an infinite rooted directed $q$-ary tree and the subgraph spanned
 by $\anc(\alpha)$ is a rooted tree, such that all edges are directed
 towards the root and the in-valency of every vertex is $r$ and the
 out-valency is $1$.
 Clearly $DL(q,r)$ is highly arc transitive.

\medskip

The digraphs $DL(q,r)$ are a directed versions of the Diestel-Leader
graphs (defined in \cite{DiestelLeader2001}) that have been studied by
various authors.  Woess \cite{Woess1992} asked if every locally finite
transitive graph is quasi-isometric to some Cayley graph.
It was conjectured by Diestel and Leader  that if
$q\neq r$ then the graph $DL(q,r)$ is
not quasi-isometric to any Cayley graph.
  This  conjecture  was proved  by Eskin,
Fisher and Whyte in \cite{EskinFisherWhyte2005}.

An optimist would hope to find a general classification of
locally finite,
highly arc transitive graphs, but
it seems very implausible that any such
classifications is possible.
But, there is a particular class of highly arc transitive digraphs
where one can give a precise description of their
structure. Surprisingly enough this particular class can be used to
probe the secrets of Willis' theory.

First, we state two simple lemmata from the paper \cite{CPW1993}
by Cameron, Praeger and
Wormald.   We prove the second one, because it is natural to apply the
permutation topology on \autga\ in the proof.

\begin{lemma}\label{LCPW}
{\rm (\cite[Proposition~3.10]{CPW1993})}  Let $\Gamma$ be a connected,
highly arc transitive digraph with finite out-valency.  Suppose
$\Gamma$ is not a directed cycle.  If $\alpha$ and $\beta$ are vertices
in $\Gamma$ and there is a directed path of length $n$ from $\alpha$
to $\beta$, then every directed path from from $\alpha$ to $\beta$ has
length $n$.  Furthermore, $\Gamma$ has no directed cycles.
\end{lemma}

\begin{lemma}\label{LTransitveLines}
Let $\Gamma$ be a  locally finite, highly arc transitive digraph.
Take two directed lines
$L_1=\ldots, \alpha_{-1}, \alpha_0, \alpha_1, \alpha_2\ldots$
and $L_2=\ldots, \beta_{-1}, \beta_0, \beta_1, \beta_2\ldots$  in $\Gamma$ then there is a an automorphism $g$ of
$\Gamma$ such that $g\alpha_i=\beta_i$ for all $i$.
\end{lemma}

\begin{proof}  Write $G=\autga$ and note that $G$ is locally compact.
Using the property that $\Gamma$ is highly arc transitive,
we can find an element $g_i\in \autga$ such that $g_i\alpha_j=\beta_j$
for all $j\in\{-i,\ldots,i\}$.  The sequence $(g_i)_{i\in {\bf N}}$ is contained in
  the set $g_1G_{\alpha_0}$, which is compact in the permutation
  topology on $\autga$.  Hence this sequence has a convergent
  subsequence that converges to an element $g$ in $\autga$ which has
  the desired property.  \end{proof}

\begin{proposition}  {\rm (\cite[Lemma~3]{Moller2002a})}
Let $\Gamma$ be a locally finite, highly arc transitive digraph and
$L$ a directed line in $\Gamma$.  Then the subdigraph $\Gamma_L$
spanned by
$\des(L)$, is highly arc transitive and has more than one end.
\end{proposition}

\begin{proof}  Write $L=\ldots, \alpha_{-1}, \alpha_0, \alpha_1,
\alpha_2\ldots$.  Consider $s$-arcs $\beta_0,\ldots, \beta_s$ and
$\gamma_0,\ldots, \gamma_s$ in $\Gamma_L$.
The vertices $\beta_0$ and $\gamma_0$
will have a common ancestor $\alpha_{i_0}$ on the line $L$.   Now we can
extend the $s$-arcs to infinite lines $L_\beta=\ldots, \beta_{-1},
\beta_0, \beta_1, \beta_2\ldots$ and $L_\gamma=\ldots, \gamma_{-1},
\gamma_0, \gamma_1, \gamma_2\ldots$ that both contain the directed ray
$\ldots, \alpha_{i_0-2},\alpha_{i_0-1}, \alpha_{i_0}$.
Then we can find an element
$g\in G$ such that $g\beta_i=\gamma_i$ for all $i$ and because $g$
maps the ray $\ldots, \alpha_{i_0-1}, \alpha_{i_0}$ into $L$ we can
see that $g(\des(L))=\des(L)$, i.e. the subdigraph $\Gamma_L$ is
invariant under $g$.  Whereupon we conclude that $\Gamma_L$ is highly
arc transitive.

Let $\beta'$ be a vertex in $\Gamma_L$.  Since $\inn_{\Gamma_L}(\beta')$ is
finite, there
clearly is a number $i$ such that $\inn_{\Gamma_L}(\beta')
\subseteq\des(\alpha_i)$.  Let $k$ be the length of a directed path
from $\alpha_i$ to $\beta'$ (by Lemma~\ref{LCPW}
all directed paths from $\alpha_i$ to $\beta'$ have the same length).
  Making use
of arc transitivity we conclude that
if $\beta\in\des_k(\alpha_0)$, then there is an element
$g\in\aut(\Gamma_L)$ such that
$g(\alpha_i)=\alpha_0$ and $g(\beta')=\beta$.
Therefore we see that if  $(\gamma,\beta)$ is an
arc in $\Gamma_L$ (i.e.\ $\gamma\in\inn_F(\beta)$)
then $\gamma\in \des(\alpha_0)$.  More precisely,
$\gamma\in \des_{k-1}({\alpha_0})$.  This is so because, if
$\alpha_0,\gamma_1,\ldots,\gamma_l,\gamma$
is a directed path from $\alpha_0$ to
$\gamma$ then $\alpha_0,\gamma_1,\ldots,\gamma_l,\gamma,\beta$ is a
directed path from $\alpha_0$ to $\alpha$ and thus has length $k$.

Set $A=\bigcup_{i\geq k}\des_i(\alpha_0)$ and $A^*=VF\setminus A$.  Suppose
$(\gamma,\beta)$ is an arc from $A^*$ to $A$.
Now $\beta\in\des_l(\alpha_0)$ for some
$l\geq k$.  Then $\gamma\in \des_{l-1}(\alpha_0)$, by  the choice of $k$.
Obviously $l=k$ and $(\gamma,\beta)$
is an arrow from $\des_{k-1}(\alpha_0)$ to $\des_{k}(\alpha_0)$.

{\em A priori}, there is also the possibility that some
arc $(\beta,\gamma)$ in
$F$ goes from $A$ to $A^*$.  But on closer look, this is impossible,
because then $\beta$ would be in $\des_l(\alpha_0)$ for some
$l\geq k$ and thus $\gamma\in\des_{l+1}(\alpha_0)\subseteq A$, and
therefore $\gamma\in A$, contradicting the assumption that $\gamma\in
A^*$.

The only arcs
between $A$ and $A^*$ are going from  $\des_{k-1}(\alpha_0)$ to
$\des_{k}(\alpha_0)$.   The set of such arcs is clearly finite (because
$\des_{k-1}(\alpha_0)$ is finite and the
out-valency of vertices in $\Gamma$ is
finite), and by removing them,
we split $\Gamma$ up into components.  The two sets
$\{\ldots,\alpha_0,\alpha_1,\ldots,\alpha_{k-1}\}$ and $\{\alpha_k,\alpha_{k+1},\ldots\}$ will
belong to different components, so we have at least two infinite components.
Hence $\Gamma$ has more than
one end.  \end{proof}  

\bigskip

The structure of digraphs like $\Gamma_L$ in the above
proposition is described in the following theorem.

\begin{theorem}
{\rm (\cite[Theorem 1]{Moller2002a})}\label{THATtree}
Let $\Gamma$ be a locally finite, highly arc
transitive digraph. Suppose that there
is a line $L= \ldots, \alpha_{-1}, \alpha_0, \alpha_1,\ldots$
such that $V\Gamma={\des}(L)$.

Then there exists a surjective homomorphism
$\phi:\Gamma\rightarrow T$ where $T$
is a directed tree
with in-valency 1 and finite out-valency.
The automorphism group of $\Gamma$ has a natural action on $T$
as a group of automorphisms such that
$\phi(g\alpha)=g\phi(\alpha)$ for every
$g\in \autga$
and every vertex $\alpha$ in $\Gamma$.
This action of $\autga$ on $T$ is highly arc
transitive. Furthermore, the
fibers  $\phi^{-1}(\alpha)$, $\alpha\in VT$,
are finite and all have the same number
of elements.
\end{theorem}

Let $\alpha$ be a vertex in a highly arc transitive digraph $\Gamma$
and denote with $c_k$ the number of vertices in $\des_k(\alpha)$.
Cameron, Praeger and Wormald in \cite[Definition~3.5]{CPW1993}
define the {\em out-spread} of a vertex
in $\Gamma$ as $\limsup_{k\rightarrow\infty} c_k^{1/k}$.   One can
define the {\em in-spread} of a highly arc transitive digraph in
a similar way.   Theorem~\ref{THATtree} implies the following

\begin{theorem}{\rm (\cite[Theorem~2]{Moller2002a})}   The out-spread of
  a locally finite, highly arc transitive digraph is an integer.
\end{theorem}

The in-spread can be used to characterize the highly arc
transitive digraphs treated in Theorem~\ref{THATtree}.

\begin{theorem}{\rm (\cite[Theorem~2.6]{MMMSTZ2005})}   Let $\Gamma$ be a
  locally finite, highly arc transitive digraph.  The in-spread of
  $\Gamma$ is 1 if and only if there is a line $L$ in $\Gamma$ such
  that $\des(L)=V\Gamma$.
\end{theorem}

\subsection{Tidy subgroups and highly arc transitive digraphs}\label{STidy}
Now we turn our attention back to totally disconnected, locally compact
groups.  The following notation will be used extensively in what
follows.  Let $G$ be  a totally disconnected, locally compact
group and $x$ a fixed element in $G$.
Take a compact open sub\-group $U$.
We set $\Omega=G/U$ and let $\alpha_0$ denote the point
in $\Omega$ that has $U$ as stabilizer.  Then define a digraph
$\Gamma=\Gamma_U$ that has $\Omega$ as a vertex set and edge set
$G(\alpha_0, x\alpha_0)$ -- the $G$-orbit of the ordered pair
$(\alpha_0, x\alpha_0)$.   Note that $\Gamma$ need not be connected.
For an integer $i$ set
$\alpha_i=x^i\alpha_0$.  The vertices $\alpha_i$ form a line $L$ in
$\Gamma$.  Observe that $x^iUx^{-i}$ is the stabilizer of $\alpha_i$
in $G$ and $U_+$ is the stabilizer of the
vertices $\alpha_0, \alpha_1, \ldots$ and $U_-$ is the stabilizer
of the vertices $\alpha_0, \alpha_{-1}, \ldots$.

\begin{proposition}\label{PTAHAT}
{\rm (Cf.~\cite[Theorem~2.1]{Moller2002})}
The subgroup $U$ satisfies condition (TA) in Definition~\ref{DTidyscale}
if and only if $G$ acts highly arc transitively on the
digraph $\Gamma$.
\end{proposition}

\begin{proof}
Let us start by looking at what happens when the digraph $\Gamma$ is
highly arc transitive.
Let $g\in U=G_{\alpha_0}$.  Since $G$ is assumed to act highly arc
transitively on $\Gamma$ and $G$ is a closed in the permutation
topology we deduce from Lemma~\ref{LTransitveLines} that $G$ acts
transitively on the set of lines in $\Gamma$.  For $i\geq 1$ we set
$\beta_i=g\alpha_i$ and let $L_1$ denote the line $\ldots,
\alpha_{-1}, \alpha_0, \beta_1, \beta_2, \ldots$.
We find an element $g_-$ that moves the line $L$ to
the line $L_1$ such that $g_-\alpha_i=\alpha_i$ for $i\leq 0$ and
$g_-\alpha_i=\beta_i$ for $i\geq 1$.  Note that $g_-\in U_-$.
Set $g_+=g_-^{-1}g$ and note that
$g_+$ fixes all the vertices $\alpha_0, \alpha_1, \ldots$ and thus
$g_+\in U_+$.  Therefore $g\in U_-U_+$.  From this we
deduce that $U$ satisfies condition TA.

Conversely, assume that $U$ satisfies condition TA.  Take a vertex $\beta$ in
$\out(\alpha_0)$.  Then there must be an element $g\in U$ such that
$g\alpha_1=\beta$.  Write $g=g_-g_+\in U_-U_+$ and we see that
$g_-\alpha_1=\beta$.  Thus $U_-$ acts transitively on
$\out(\alpha_0)$.
Now we use induction over $s$ to show that $G$
acts transitively on the set of $s$-arcs.  Suppose we are given an
 $(s+1)$-arc $\beta_0, \ldots, \beta_s, \beta_{s+1}$.  Use the induction
hypothesis to find an element $h\in G$ such that $h\alpha_0=\beta_s,
h\alpha_{-1}=\beta_{s-1},\ldots, h\alpha_{-s}=\beta_0$.  Then by the above,
$hU_+h^{-1}$ acts transitively on $\out(h\alpha_0)=\out(\beta_s)$.  We
pick an element $h'$ from $hU_+h^{-1}$ such that
$h'(h\alpha_1)=\beta_{s+1}$.  Now we have found an element $h'h$ that
moves the $(s+1)$-arc  $\alpha_0, \ldots, \alpha_s, \alpha_{s+1}$
to the $(s+1)$-arc $\beta_0, \ldots, \beta_s, \beta_{s+1}$ and can
conclude that $G$ acts transitively on the $(s+1)$-arcs in $\Gamma$ and
also that $G$ acts highly arc transitively on $\Gamma$.
\end{proof}

\bigskip

Condition (TB) can also be translated in to a condition about the graph
$\Gamma$ defined at the start of the section.
We use the following lemma.

\begin{lemma}{\rm (\cite[Lemma 3]{Willis1994})}
Let $G$ be a totally disconnected,
  locally compact group and $x\in G$.  Suppose that $U$ is a compact,
  open subgroup of $G$ that satisfies condition (TA).  Then

(a) $U_{++}$ is closed if and only if $U_{++}\cap U=U_+$.

(b) $U_{++}$ is closed if and only if $U_{--}$ is closed.
\end{lemma}

In our setting $U_{++}$ is the set of all elements $g$ in $G$ such that
there exists a number $k$ such that $g$ fixes $\alpha_k,
\alpha_{k+1},\ldots$.  The condition that $U_{++}\cap U=U_+$ says that
an element in $G$ that fixes $\alpha_0$ and also $\alpha_k,
\alpha_{k+1}, \ldots$ for some $k\geq 0$ must also fix $\alpha_1,
\ldots, \alpha_{k-1}$.  If we assume that $G$ acts highly arc
transitively on $\Gamma$ then this implies that $\alpha_0, \ldots,
\alpha_k$ is the unique path in $\Gamma$ from $\alpha_0$ to
$\alpha_k$ and we conclude that the subgraph spanned by
$\des(\alpha_0)$ is a tree.

On the other hand, if the subgraph spanned
by $\des(\alpha_0)$ is a tree then clearly a group element that fixes
$\alpha_0$ and $\alpha_k$ must fix $\alpha_1,
\ldots, \alpha_{k-1}$ since these vertices lie on the directed
path from $\alpha_0$ to $\alpha_k$.  Hence  $U_{++}\cap U=U_+$.
Thus we have shown the following result.

\begin{proposition}
Suppose $U$ satisfies condition (TA).  Then $U$ satisfies condition (TB)
if and only if the subgraph spanned by $\des(\alpha_0)$ is a tree.
\end{proposition}

Putting these observation together as a theorem we get.

\begin{theorem}
 {\rm (Cf.~\cite[Theorem~3.4]{Moller2002})}
Let $G$ be a totally disconnected, locally compact group and $U$ a
compact, open subgroup.  Let $x$ be an element in $G$ and define a
graph $\Gamma$ such that the vertex set is $G/U$ and the edge set is
$G(\alpha_0, x\alpha_0)$ where $\alpha_0$ is the vertex in $\Gamma$ such
that $U=G_\alpha$.  Suppose furthermore that the orbit of $\alpha_0$ under $x$ is infinite.  Then $U$ is tidy for $x$ if and only if $G$ acts
highly arc transitively on $\Gamma$ and the subgraph spanned by
$\des(\alpha_0)$ is a tree.
\end{theorem}

\subsection{Using the connection}
In this section $G$ denotes a totally disconnected, locally compact group.
From the definition of tidy subgroup it is far from obvious that there
always is a compact, open subgroup of $G$ that is tidy for a given
element $x$ in $G$.  Our first task is thus to construct
a compact, open subgroup $U$ that is tidy for $x$.

First, the case where $x$ is periodic (i.e.~the subgroup generated
by $x$ is relatively compact).  Let $U$ be a compact, open
subgroup.  Put $\Omega=G/U$. Let $\alpha$ be a point in $\Omega$ such
that $G_\alpha=U$.   Define $A$ as the closure (in the given topology
on $G$)  of the subgroup generated by $x$.  By assumption $A$ is
compact.  Since the permutation topology induced by the action of $G$
on $\Omega$ is contained in the original topology on $G$ we conclude
that $A$ is also compact in the permutation topology.  Hence,
by Lemma~\ref{LCompact}(ii), all the
orbits of the subgroup generated
by $x$ are finite.  So there is a number $N$ such that
$x^N\alpha=\alpha$ and, therefore,  $x^N\in G_\alpha=U$.
The subgroup $U\cap xUx^{-1}\cap \cdots\cap x^{N-1}Ux^{-(N-1)}$ is compact and
open and normalized by $x$ and thus tidy for $x$.  Hence we will
assume in what follows that $x$ is not periodic.

Let $V$ be some compact, open subgroup of $G$.  Construct a graph
$\Gamma=\Gamma_V$ as done at the start of the last section.  From the proof of
Proposition~\ref{PTAHAT} we see that $G$ acts highly arc transitively
on $\Gamma$ if and only if $V_-$ acts transitively on $\out(\alpha_0)$.
Look at the group $V_n=\bigcap_{i=0}^{n}x^{-i}Vx^i=G_{\alpha_0,
  \alpha_{-1},\ldots, \alpha_{-n}}$.  We claim that there is a number $n$
 such that $V_n\alpha_1=V_-\alpha_1$.  Otherwise one could find an
element $g_i\in V_i$ for each $i$ such that
$g_i\alpha_1\nin V_-\alpha_1$.  The
sequence $(g_i)_{i\in{\bf N}}$
has a convergent subsequence converging to an element
$g$ and clearly this element is in $V_-$, but $g\alpha_1\nin
V_-\alpha_1$, so we have reached a contradiction.  Now set $W=V_n$.
Note that $W_+=V_+$.  We can use a similar argument as
in the first part of Proposition~\ref{PTAHAT} to show that $W$
satisfies condition (TA).   Using this compact, open subgroup $W$ to get
a compact, open subgroup that also satisfies condition (TB) is more
involved, and the details will be left out.  By finding a compact, open
subgroup $W$ satisfying (TA),
we have ensured that $G$ acts highly arc transitively on the digraph
$\Gamma_W$.  What is missing is condition (TB), which would mean that the
subgraph spanned by the descendants of a vertex is a tree.  To achieve
that, Theorem~\ref{THATtree} is used to produce a highly arc
transitive digraph, wherein the graph spanned by the descendants of a
vertex is a tree.  This will then prove the following theorem of Willis.

\begin{theorem}{\rm (\cite[Theorem 1]{Willis1994},
see also \cite[Theorem 4.1]{Moller2002})}
Let $G$ be a totally disconnected, locally compact group and $x$ an
element of $G$.  Then there is a compact, open subgroup $U$ of $G$ that
is tidy for $x$.
\end{theorem}

Now we have ensured that there is something to talk about.
We next use digraphs to deduce further facts about tidy subgroups
and the scale function.
First, we use Lemma~\ref{LModular} to deduce
the following.

\begin{theorem}{\rm (\cite[Corollary 1]{Willis1994})}\label{Tscale-modular}
Let $G$ be a totally disconnected, locally compact group.  Denote by
$\Delta$ the modular function on $G$ and by ${\bf s}$ the scale
function on $G$.  Then, for every $x\in G$,
$$\Delta(x)=\frac{{\bf s}(x)}{{\bf s}(x^{-1})}.$$
\end{theorem}

\begin{proof}  Let $U_1$ and $U_2$ be compact open subgroups of $G$ such
that
$$|U_1:U_1\cap x^{-1}U_1x|={\bf s}(x)\qquad\mbox{and}
\qquad |U_2:U_2\cap xU_2x^{-1}|={\bf s}(x^{-1}).$$
Note that
$$|U_1:U_1\cap xU_1x^{-1}|\geq {\bf s}(x^{-1})
\qquad\mbox{and}\qquad
|U_2:U_2\cap x^{-1}U_2x|\geq {\bf s}(x).$$
Now we use the Remark following Lemma~\ref{LModular} and get
$$\frac{{\bf s}(x)}{{\bf s}(x^{-1})}\geq
\frac{|U_1:U_1\cap x^{-1}U_1x|}{|U_1:U_1\cap xU_1x^{-1}|}
=\Delta(x)=\frac{|U_2:U_2\cap x^{-1}U_2x|}{|U_2:U_2\cap xU_2x^{-1}|}
\geq \frac{{\bf s}(x)}{{\bf s}(x^{-1})}.$$
Hence $\Delta(x)={\bf s}(x)/{\bf s}(x^{-1}).$   \end{proof}

\medskip

This also implies the following corollary.

\begin{corollary}{\rm (\cite[Corollary 3.11]{Willis2001a})}
\label{Cequal}
Let $x$ be an element of a totally disconnected, locally compact group $G$,
and  $U$  a compact, open subgroup of $G$.
Then $|U:U\cap x^{-1}Ux|={\bf s}(x)$ if and
only if $|U:U\cap xUx^{-1}|={\bf s}(x^{-1})$.
\end{corollary}

Tidy subgroups are related to the scale function as described in
Theorem~\ref{TScaletidy}.  The proof of  Theorem~\ref{TScaletidy} is
involved, and we will only have a look at the proof that compact, open
subgroup $U$ such that ${\bf s}(x)=|U:U\cap x^{-1}Ux|$ must be tidy.

Consider a compact, open subgroup $U$, a fixed element $x\in G$, and the
digraph $\Gamma$ defined above.  By the above, when trying to minimize
$|U:U\cap x^{-1}Ux|$ in order to find $\s(x)$, we could equally try to
minimize $|U:U\cap xUx^{-1}|$.  The latter index is just the
out-valency in the digraph $\Gamma_U$.  When constructing a tidy
subgroup for $x$, we start with an arbitrary compact, open subgroup $V$
and next find a compact, open subgroup $V$ satisfying (TA). The
out-valency in $\Gamma_V$ is at most the out-valency of $\Gamma_U$.
In the second step, we ensure that condition (TB) is satisfied, and
in the process the out-valency does not increase.  Thus we can be sure
that if $U$ minimizes $|U:U\cap xUx^{-1}|$ then $U$ must be tidy for
$x$.

\bigskip

Again we look at a compact open subgroup $U$ and the graph $\Gamma$ as
above.  Note
that $U\alpha_n=|U:U\cap x^nUx^{-n}|$.  If $\Gamma$ is highly arc
transitive, then this is precisely the number $b_n$
of vertices $\beta$ such that
there is a directed path of length $n$ from $\alpha_0$ to $\beta$.
The out-valency $d_+$ of $\Gamma$ is equal to $|U:U\cap xUx^{-1}|$.
The subgraph spanned by $\des(\alpha_0)$ is a tree if and only if
 $b_n=d_+^n=|U:U\cap xUx^{-1}|^n$ for all natural numbers $n$.  But $U$ is tidy
if an only if $\Gamma$ is highly arc transitive and the subgraph
spanned by $\des(\alpha_0)$ is a tree. Thus we derive
the following result.

\begin{theorem}
 {\rm (\cite[Corollary~3.5]{Moller2002})}
Let $G$ be a totally disconnected, locally compact group and $x$ an
element in $G$.  Then a compact, open subgroup $U$ is tidy for $x$ if
and only if
$$|U:U\cap x^nUx^{-n}|=|U:U\cap xUx^{-1}|^n$$
for all $n\geq 1$.
\end{theorem}

\begin{corollary}
Let $G$ be a totally disconnected locally compact group and $x$ an
element in $G$. Then $\s(x^n)=\s(x)^n$.
\end{corollary}

\begin{proof}  Let $U$ be a compact, open subgroup of $G$ that is tidy
for $x$.
It is easy to check that if a subgroup $U$ is tidy for
$x$ then $U$ is also tidy for $x^n$ for every integer $n$.
Hence
$$\s(x^n)=|U:U\cap x^{-n}Ux^{n}|=|U:U\cap x^{-1}Ux|^n=\s(x)^n. \qedhere$$
\end{proof}

\bigskip

If $\Gamma$ is highly arc transitive, the index $|U:U\cap x^{-n}Ux^{n}|$
is the number of vertices $\beta$ such that
$\alpha_0$ is in $\des_n(\beta)$. This observation suggests that we compare
the scale function and the in-spread of the associated digraph $\Gamma$. The following
theorem describes their relationship.

\begin{theorem} {\rm (\cite[Theorem~7.7]{Moller2002})}
Let $G$ be a totally disconnected, locally compact group and $x$ an
element of $G$.  If $V$ is some compact, open subgroup of $G$, then
$$\s(x)=\lim_{n\rightarrow\infty}|V:V\cap x^{-n}Vx^n|^{1/n}.$$
\end{theorem}

For a different formulation and a proof see
\cite[Lemma~4]{BaumgartnerWillis2006}.
This line of thought also gives us information about the case
$\s(x)=1$.

\begin{theorem}
{\rm (\cite[Corollary~7.8]{Moller2002})}
Let $G$ be a totally disconnected, locally compact group and $x$ an
element of $G$ such that $\s(x)=1$.
If $V$ is some compact, open subgroup of $G$, then there is a constant
$C$ such that
$|V:V\cap x^{-n}Vx^n|\leq C$ for all $n\geq 0$.
\end{theorem}

These two results can also be formulated as results about
permutation groups.

\begin{theorem}\label{TExponential}
Let $G$ be a group acting transitively on a set $\Omega$.  Assume that
all sub\-orbits of $G$ are finite.  Let $x$ be an element in $G$ and
$\alpha_0$ a point in $\Omega$.  Set $\alpha_i=x^i\alpha_0$.  Then
either there is a constant $C$ such that $|G_{\alpha_0}\alpha_n|\leq
C$ for all $n$ or the numbers $|G_{\alpha_0}\alpha_n|$ grow
exponentially with $n$ and
$\lim_{n\rightarrow\infty}|G_{\alpha_0}\alpha_n|^{1/n}=s$
for some integer $s$.
\end{theorem}

{\em Remark.}  
In
\cite{Trofimov2007}  Trofimov studies {\em
  generalized $x$-tracks}, which
are similar to  the directed line $...,
\alpha_{-1}, \alpha_0, \alpha_1, \alpha_2, ...$ that is fundamental to
the graph-theoretical interpretation of Willis' theory in
\cite{Moller2002}.  Theorem~\ref{TExponential} is 
clearly related to \cite[Theorem~4.1, part 3]{Trofimov2007}.

\bigskip

The final illustration of the uses of graphs in Willis' structure
theory is a proof of the
following theorem.

\begin{theorem} {\rm (\cite[Theorem 2]{Willis1995})}
Let $G$ be a totally disconnected, locally compact group.
The set $P(G)$ of periodic elements in $G$ is closed.
(An element $x\in G$ is {\em periodic} if and only if $\overline{\langle
  x\rangle}$ is compact.)
\end{theorem}

\begin{proof}  The trick is to use the fact that
that a connected infinite and locally finite highly arc
    transitive digraph has no directed cycles, see Lemma~\ref{LCPW}.

    Suppose $x$ is not periodic, but is in the closure of $P(G)$.
Let $U$ be a compact, open subgroup of $G$, that is tidy for
    $x$.  Define a digraph $\Gamma$ as at the start of
 Section~\ref{STidy}.  If $x$ is not periodic, then the orbit of
    $\alpha_0$ under $x$ is infinite, and the connected
    component of $\Gamma$ that contains $\alpha_0$ is infinite.  (It
    must contain the line $\ldots,\alpha_{-1}, \alpha_0, \alpha_1,
    \alpha_2,\ldots$.)
    The set $xU$ is an open neighbourhood of $x$, and must therefore
    contain some periodic element $g$.  The fact that $g\in xU=xG_{\alpha_0}$
implies
$g\alpha_0=x\alpha_0=\alpha_1$.  The element $g$ is periodic, hence
the orbit of $\alpha_0$ under $g$ is finite, and therefore there is an
integer $n$ such that $g^n(\alpha_0)=\alpha_0$.  The sequence
$\alpha_0, \alpha_1=g\alpha_0, \beta_2=g^2\alpha_0,
\ldots, \beta_n=g^n\alpha_0=\alpha_0$
is a directed cycle in
$\Gamma$.  This contradicts Lemma~\ref{LCPW}  mentioned above.
Hence we
conclude that it is impossible that the closure of $P(G)$ contains any
elements that are not periodic.  Thus $P(G)$ is closed.  \end{proof}

\section{Rough Cayley graphs}\label{SRough}

Most of the material in this section is taken from a paper by Kr\"on
 and M\"oller \cite{KronMoller2008}.   Let $G$ be a compactly
 generated, totally disconnected, locally compact group.  In
 \cite{KronMoller2008} the authors construct a locally finite,
 connected graph with a transitive $G$-action, whose vertex
 stabilizers are compact, open subgroups of $G$. This graph is called
 a {\em rough Cayley graph} of $G$. As demonstrated in
 \cite{KronMoller2008}, and summarized in this section, 
a rough Cayley graph can be used to study compactly generated, 
locally compact groups in a similar way as an ordinary Cayley 
graph is used to study a finitely generated group.
In this article, we illustrate this approach, by using rough 
Cayley graphs to generalize the concept of ends of groups and 
to study compactly generated, locally compact groups of 
polynomial growth. 

Below, we explain how to construct a rough Cayley graph and 
it is also shown that any two rough Cayley graphs for a given group
are quasi-isometric.
The applications of the rough Cayley graph to the theory of ends 
of groups and to groups of polynomial growth are only sketched; 
details can be found in \cite{KronMoller2008}, 
where tools from \cite{Dunwoody1982} and \cite{DicksDunwoody1989} are 
used extensively.

\subsection{Definition of a rough Cayley graph}
\label{SDefRough}
\begin{definition}{\rm (\cite[Definition~2.1]{KronMoller2008})}
Let $G$ be a topological group.  A connected graph $\Gamma$ is said to be a
{\em rough Cayley graph} of $G$ if $G$ acts transitively on $\Gamma$ and
the stabilizers of vertices are compact, open subgroups of $G$.
\end{definition}

In this section we show that if $G$ is a compactly
generated locally compact group that contains a compact open subgroup
then $G$ has a locally finite rough Cayley graph and any two rough
Cayley graphs are
quasi-isometric to each other.  The approach here is different from
the approach in \cite{KronMoller2008}.

Let $G$ be a compactly generated topological group.  For a compact
generating set $S$ we form the Cayley graph   $\Gamma=\Cay(G, S)$ of $G$ with
respect to $S$.  The vertex set of $\Gamma$ is equal to $G$ and thus
carries a topology.  The compactness of $S$ and the continuity of
multiplication in $G$ implies that if $A$ is a relatively compact set
of vertices in $\Gamma$ then the neighbourhood of $A$ in $\Gamma$ is
contained in the set $A\cdot S$ and is thus relatively compact also.
This can be  used 
to prove that a set $A$ of vertices in $\Gamma$ is relatively compact
if and only if it has finite diameter in the graph metric on $\Gamma$,
see \cite[2.3 Heine-Borel-Eigenschaft]{Abels1974}.

\begin{definition}  \label{Dorbit}
Let $G$ be a group acting transitively on a connected
  graph $\Gamma$.  Suppose $U$ is a subgroup of $G$ that contains the
  stabilizer of some vertex $\alpha$.  The orbit $U\alpha$ is a block
  of imprimitivity.  Let $\Gamma_U$ denote the quotient graph
  with respect to the $G$-congruence whose classes are the translates
  under $G$ of the set $U\alpha$.
\end{definition}

\begin{lemma} \label{Llocallyfinite}
Let $G$ be a compactly generated topological group.  Suppose $S$
is a compact generating set and
$U$ is a compact open subgroup of $G$.
Then the graph
$\Cay(G,S)_U$ is locally finite.
\end{lemma}

\noindent
\begin{proof}  The neighbourhood of a coset $gU$ in $\Gamma$ is compact
and can thus be covered by finitely many right cosets of $U$.  Whence
$\Gamma_U$ is locally finite.  \end{proof}

\medskip

Lemma~\ref{Llocallyfinite} shows that every compactly generated group $G$
that has a compact, open subgroup $U$, has a locally finite rough Cayley
graph, namely, $\Cay(G,S)_U$.

The proof of a result of Sabidussi \cite[Theorem 2]{Sabidussi1964}, restated
below in our own terminology, can be used to show that every locally finite
rough Cayley graph for a compactly generated group with a compact, open
subgroup can be obtained in this fashion.

\begin{theorem}{\rm (Cf.~\cite[Theorem 2]{Sabidussi1964})}
Let $\Gamma$ be a connected graph and $G$ a transitive subgroup of
$\aut(\Gamma)$.  Then there is a set $S$ of generators of $G$ such
that $\Gamma\cong \Cay(G, S)_U$ where $U$ is the stabilizer in $G$
of some vertex in $\Gamma$.
\end{theorem}

In our setting we are thinking of a topological group $G$ acting on a
locally finite  rough Cayley graph $\Gamma$ such that stabilizers of vertices are compact open subgroups but the action is  not necessarily
faithfully.  Take a
vertex $\alpha$ in $\Gamma$ and let $U$ denote the stabilizer in $G$
of $\alpha$.  Looking at Sabidussi's proof of his theorem we deduce that if
$S$ is the union of $U$ and all elements $h$ in $G$ such that
$\alpha$ and $h\alpha$ are adjacent in $\Gamma$
then $\Gamma\cong \Cay(G, S)_U$.
Note also that the orbit $S\alpha$ consists precisely of
$\alpha$ and  all
its neighbours.  The graph $\Gamma$ is by assumption locally finite
so $S$ is a finite
union of cosets of $U$ and thus compact.

The influential concept of
quasi-isometry was introduced by Gromov \cite{Gromov1986} and has been
widely used since.   

\begin{definition}
Two metric spaces $(X, d_X)$ and $(Y, d_Y)$ are said to be {\em
  quasi-isometric} if there is a map $\varphi:X\rightarrow Y$ and 
constants $a\geq 1$ and $b\geq 0$ such that for all points
  $x_1$ and $x_2$ in $X$
\[a^{-1}d_X(x_1, x_2)-a^{-1}b\leq d_Y(\varphi(x_1), \varphi(x_2))
\leq ad_X(x_1, x_2)+ab,\]
and for all points $y\in Y$ we have
\[d_Y(y, \varphi(X))\leq b.\]
A map $\varphi$ between two metric spaces satisfying the above
conditions is called a {\em quasi-isometry}.
\end{definition}

Two connected graphs $X$ and $Y$ are called quasi-isometric if $(VX,
d_X)$ and $(VY, d_Y)$ are quasi-isometric. 
Being quasi-isometric is an equivalence relation on the class of
metric spaces. 

\begin{theorem}\label{Tquasi}
Let $G$ be a compactly generated group.  Assume that $G$ admits a rough Cayley graph.
All rough Cayley graphs for $G$ are quasi-isometric.
\end{theorem}

The first step in the proof of Theorem \ref{Tquasi} is the following
Lemma. The proof
of the Lemma depends on the {\em Heine-Borel Eigenschaft},
\cite[2.3 Heine-Borel-Eigenschaft]{Abels1974}, mentioned in the
paragraph preceding Definition~\ref{Dorbit}.

\begin{lemma} \label{Lquasi}
(i) Let $G$ be a compactly generated topological group.  Suppose $S_1$ and
$S_2$ are compact generating sets for $G$.  Then the Cayley-graphs
$\Cay(G, S_1)$ and $\Cay(G, S_2)$ are quasi-isometric.

(ii) Suppose $S$ is a
compact generating set and
$U$ is a compact subgroup of $G$.
Then the graphs
$\Cay(G, S)$ and $\Cay(G, S)_U$ are quasi-isometric.
\end{lemma}

\noindent
\begin{proof}  (i)  There is a constant $C$, such that the elements in $S_1$ can be
expressed as words of length $\leq C$ in the elements in $S_2$
and {\em vice versa}. For each element of $S_1$ respectively $S_2$
fix a word in $S_2$ respectively $S_1$ with this property.
Using these correspondences, words in $S_1$ and $S_2$ may be
expressed as words in the other set, that are at most $C$ times longer.
This shows that the identity map on $G$ extends to quasi-isometries
$\operatorname{Cay}(G, S_1)\to \operatorname{Cay}(G, S_2)$ and
$\operatorname{Cay}(G, S_2)\to \operatorname{Cay}(G, S_1)$
that are inverse to each other.

(ii) By part (i)  we know that $\Cay(G, S)$ and
$\Cay(G, S\cup U)$ are quasi-isometric.  Thus we may assume that $S$
contains $U$.  Under this assumption, each of the right cosets of $U$ has therefore diameter 1.
The quotient graph $\Cay(G, S)_U$ is obtained by contracting each of these
cosets.  Clearly this operation preserves quasi-isometry and the claim follows.   \end{proof}

\medskip

\begin{proof} (Theorem~\ref{Tquasi})  Suppose $\Gamma_1$ and
$\Gamma_2$ are rough Cayley graphs of $G$.  Then we can find compact
generating sets $S_1$ and $S_2$ and compact, open subgroups $U_1$ and
$U_2$ such that $\Gamma_1=\Cay(G, S_1)_{U_1}$ and
$\Gamma_2=\Cay(G, S_2)_{U_2}$.  By part~(ii) of Lemma~\ref{Lquasi}, $\Gamma_1$ is quasi-isometric to
	$\Cay(G, S_1)$, which in turn is quasi-isometric to  $\Cay(G,
	S_2)$ by part~(i) of Lemma~\ref{Lquasi}, which is again
	quasi-isometric to $\Gamma_2=\Cay(G, S_2)_{U_2}$ 
by part~(ii) of the same result.
\end{proof}

Suppose $G$ is a locally compact group with a compact, open subgroup.
If the group is compactly generated, Lemma \ref{Llocallyfinite} gives
the existence of a locally finite rough Cayley graph. Conversely,
the existence of a rough Cayley graph implies that $G$ is compactly
generated.

\begin{proposition} {\rm (\cite[Corollary~1]{Moller2003})}
\label{Pcompactgenerated}
Suppose that $G$ is a locally compact topological group and that
$G$ acts transitively on a connected, locally finite graph such that
the stabilizers of vertices in $G$ are compact, open subgroups of $G$.
Then $G$ is compactly generated.
\end{proposition}

The proof of the proposition is based on the Lemma below.

\begin{lemma} \label{Ltransitive}
Let $G$ be a group acting transitively on a connected, locally finite
graph $\Gamma$.  Then $G$ has a finitely generated, transitive subgroup.
\end{lemma}

\noindent
\begin{proof} Fix a reference vertex $\alpha$.  Denote the neighbours
of $\alpha$ by $\beta_1, \ldots, \beta_n$.  Choose elements $h_1,
\ldots, h_n$ such that $h_i\alpha=\beta_i$.  We claim that
$H=\langle h_1, \ldots, h_n\rangle$ is transitive on the vertices of
$\Gamma$.   Note that all the vertices in $\Gamma$ that are adjacent
to the vertex $\alpha$ are in the $H$-orbit of $\alpha$.  Suppose that
$\beta=h\alpha$ for some $h\in H$.   Then $hh_1h^{-1}\beta, \ldots,
hh_nh^{-1}\beta$ is an enumeration of all the neighbours of $\beta$.
Whence the neighbours of $\beta$ are also contained in the $H$-orbit
of $\alpha$.  Since our graph is assumed to be connected, we conclude
that $H$ acts transitively on the vertices.      \end{proof}

Proposition~\ref{Pcompactgenerated}
follows from Lemma \ref{Ltransitive}, because the union of the
stabilizer of a vertex with a finite generating set for a transitive subgroup
forms a compact generating set
for $G$.

The following theorem concludes our basic considerations of rough Cayley graphs.
Its first part is well known.

\begin{theorem} \label{TCocompact}
{\rm (\cite[Corollary~2.11]{KronMoller2008})}
Let $G$ be a compactly generated topological group  that
has a compact open subgroup.  Assume that $H$ is a cocompact closed
subgroup of $G$.  Then $H$ is compactly generated and if $\Gamma_G$ is a
rough Cayley graph for $G$ and $\Gamma_H$ is a rough Cayley graph for $H$
then $\Gamma_G$ and $\Gamma_H$ are quasi-isometric.
\end{theorem}

\begin{proof}  Let $X$ be a locally finite rough Cayley graph for $G$.  By
Lemma~\ref{Lcocompact} we know $H$ acts with finitely many orbits on $X$.
Choose a vertex $\alpha$ in $\Gamma$.
Then there is a number $k$ such that every vertex in $\Gamma$ is in
distance at most $k$ from the orbit $H\alpha$.   Now form the graph
$\Gamma'$ which has the same vertex set as $\Gamma$ but two
vertices being
adjacent if and only if their distance in $\Gamma$ is at most $2k+1$.
Note that $\Gamma'$ is also
locally finite.  Consider the subgraph
$\Delta$ of $\Gamma'$
spanned by the vertices in $H\alpha$.  Suppose $\alpha$ and $\beta$
are vertices in $\Delta$ and that $\alpha=\alpha_0,
\alpha_1, \ldots, \alpha_n=\beta$ is a path in
$\Gamma$ (and thus also a path in $\Gamma'$) from $\alpha$
to $\beta$.  For each $\alpha_i$
choose a vertex $\beta_i$ in $H\alpha$ such that
$d_\Gamma(\alpha_i, \beta_i)\leq k$.
Then $d(\beta_i, \beta_{i+1})\leq 2k+1$ so either
$\beta_i=\beta_{i+1}$
or $\beta_i$ and
$\beta_{i+1}$ are adjacent in $\Gamma'$.  Whereupon we conclude that
$\Delta$ is
connected.  The action of $H$ on the connected, locally finite graph $\Delta$ 
is transitive with compact, open vertex-stabilizers. Hence $\Delta$ 
is a rough Cayley graph for $H$ and $H$ is compactly generated 
by Proposition~\ref{Pcompactgenerated}.
  From the above it is clear that $\Delta$ is
quasi-isometric to $\Gamma$.  The second part of the theorem
follows by Theorem~\ref{Tquasi}.  \end{proof}

{\em Remark.}  For the rest of Section~\ref{SRough} we will focus on compactly generated, totally disconnected, locally compact groups.
Corresponding results hold for compactly generated, locally compact groups that contain a compact, open subgroup.

\subsection{Application to FC$^-$-groups}
As mentioned in Section~\ref{STrofimov}, an element $g$ of a topological
group $G$ is called a FC$^-$-element if its conjugacy class in $G$
has compact closure.  It is an easy exercise to show that the
FC$^-$-elements of $G$ form a normal subgroup $B(G)$ of $G$. If $G=B(G)$
then $G$ itself is called a FC$^-$-group.  These concepts have been
extensively studied, see for example the paper
\cite{GrosserMoskowitz1971} by Grosser and Moskowitz and various
papers by Wu and his collaborators, e.g. \cite{WuYu1972} and
\cite{Wu1991}.

A basic question about the subgroup $B(G)$ is whether it is closed or
not.  This question was discussed by Tits in \cite{Tits1964}, where it
proved that $B(G)$ is closed if $G$ is a connected locally compact
group.   But, Tits also gives an example of a locally compact, totally
disconnected group where $B(G)$ is not closed.    Below is another
example of such a group described by using graphs.

\medskip

{\em Example.}  Let $\Gamma$ denote a directed tree such that each
vertex has in-valency 1 and out-valency 2.  Choose a directed line
$\ldots, \alpha_{-1}, \alpha_0, \alpha_1, \alpha_2, \ldots$ in $\Gamma$.
We say that vertices $\alpha$
and $\beta$ are in the same {\em horocycle} if there is a number $n$
such that the unique path in $\Gamma$ from $\alpha$ to $\beta$
starts by going backwards along
$n$ arcs and then going forward along $n$ arcs.
(This concept is
also discussed in Section~\ref{Shat}.) Membership in the same
horocycle is an equivalence  
relation on vertices. We denote by $C_i$ the equivalence
class of $\alpha_i$.
Define $H$ as the subgroup of $\autga$ that stabilizes $C_0$.  Let $G$
denote the permutation group that $H$ induces on $C_0$ and endow $G$
with the permutation topology arising from that action.  We think of
$C_0$ as a metric space, the metric being the restriction of the graph
metric on $\Gamma$.  If $g$ is an
element of $G$ that fixes all but finitely many vertices in $C_0$ then
by Lemma~\ref{LWoessBounded} the element $g$ is an FC$^-$-element of
$G$.  It is also clear that if $g$ and $h$ are elements of $G$, both
with finite support, then there is a conjugate of $h$ whose support is
disjoint from the support of $g$.  Let $g_i$ be an element if $G$ with
finite support such that there is a vertex $\alpha\in C_0$ such that
$d(\alpha, g_i \alpha)=2i$.  (We could define $g_i$ explicitly by
defining $\beta_{i+1}$ as the vertex such that $(\alpha_{-i-2},
\beta_{i+1})$ is an arc in $\Gamma$ and $\beta_{i+1}\neq \alpha_{-i-1}$
and then let $g$ transpose the two outward directed arcs from
$\beta_{i+1}$.)   We may assume that $\supp g_i\cap \supp g_j=\emptyset$ if
$i\neq j$.  All the $g_i$'s are contained in $B(G)$ and $h_i=g_1\ldots
g_1$ is also in $B(G)$.   Then
we define $g$ as the limit of the sequence of the $h_i$'s ($g\alpha=h_i\alpha$
if $\alpha$ is in $\supp h_i$ and $g\alpha=\alpha$ if $\alpha$ is not
in the support of any of the $h_i$'s).  Clearly $g$ is not a bounded
isometry of $C_0$ and thus $g$ is not in $B(G)$.  Hence $B(G)$ is not closed.
Indeed, one can easily show that $B(G)$ is dense in $G$.

\bigskip

The construction of a rough Cayley graph can be used in the study
of FC$^-$-elements in groups.  Trofimov in \cite{Trofimov1985a} proved
the following.

\begin{theorem}{\rm (\cite{Trofimov1985a}, cf. \cite[Theorem~3]{Woess1992})}
Let $\Gamma$ be a vertex transitive, connected, locally finite graph.  The
subgroup of bounded automorphisms in $\autga$ is closed in $\autga$.
\end{theorem}

Using a rough Cayley graph for the group and this theorem allows us
to prove the following result.

\begin{theorem}{\rm (\cite[Theorem~2]{Moller2003})}  Let $G$ be a
  compactly generated totally disconnected locally compact group.
  Then the subgroup of FC$^-$-elements is closed in $G$.
\end{theorem}

The proof is simple.  One constructs a locally finite rough Cayley
graph for $G$.  The action of $G$ on $\Gamma$ gives a continuous
homomorphism $G\rightarrow\autga$.  It is easy to show that the kernel
of this homomorphism is compact and the image is closed in $\autga$ (see
\cite[Corollary~1]{Moller2003}).  The subgroup of bounded automorphism
in $\autga$ is
closed by the theorem of Trofimov above,
and the subgroup of FC$^-$-elements is closed in $G$, because it
is the preimage of the subgroup of bounded automorphisms in
$\autga$.

\subsection{Rough ends}
The space of ends of a connected, locally finite graph $\Gamma$ is the 
boundary of a certain compactification of $V\Gamma$, where we think of
$V\Gamma$ as having the discrete topology. One can 
think of different ends as representing the "different ways of 
going to infinity" in $\Gamma$.  Ends of graphs can both be defined by
using topological concepts and by purely graph theoretical means.
The graph theoretical method extends to graphs that are not locally
finite, but then the ends do not give a compactification of the vertex
set like in the locally finite case.

Recall that a {\em ray} in a graph is a sequence $\alpha_0, \alpha_1,
\ldots$ of
distinct vertices such that $\alpha_i$ is adjacent to $\alpha_{i+1}$
for all $i$.
The graph theoretic approach is to define the
ends of a graph $\Gamma$ as equivalence classes
of rays.  Two rays are equivalent
(i.e.~are in the same {\em end}) if
there is the third ray that intersects both infinitely often.  This
definition goes back to the paper \cite{Halin1964} by Halin.

Ends can also be defined with reference to connected components when
finite sets of edges are removed from the graph.  For a finite set
$\Phi$ of edges define ${\cal C}_\Phi$ as the set of connected
components of $\Gamma\setminus \Phi$.
Suppose $\Phi_1$ and
$\Phi_2$ are two finite sets of edges such that
$\Phi_1\subseteq\Phi_2$.
There is a natural map ${\cal C}_{\Phi_2}\rightarrow {\cal
  C}_{\Phi_1}$ that takes a component $c$ of $\Gamma\setminus \Phi_2$ to
the component of $\Gamma\setminus \Phi_1$ that contains $c$.  The collection of all the sets ${\cal C}_\Phi$ where
$\Phi$ ranges over all finite sets of edges in $\Gamma$, together
with the connecting natural maps forms an
inverse system.  The space of ends $\Omega\Gamma$ of $\Gamma$
(with the natural topology of the
inverse limit) is then defined as the inverse
limit of this system.  One could also look at the components of
$\Gamma\setminus \Phi$ where $\Phi$ is a finite set of vertices, but
when the graph is locally finite the inverse limits are homeomorphic.
 This approach goes back to the thesis of Freudenthal in 1931.  In later works
Freudenthal and Hopf built up a theory of ends of spaces, see
\cite{Freudenthal1942}, \cite{Freudenthal1945} and \cite{Hopf1944}.

The graph theoretic approach to ends also leads to a natural definition of a topology.  The {\em co-boundary} of a set $c\subseteq V\Gamma$ is the set of all edges in
 $\Gamma$ such that one of its end vertices is in $c$ and the other is in
 $V\Gamma\setminus c$.  Denote the co-boundary of $c$ with $\delta c$.
Suppose $|\delta c|<\infty$.
One sees that if $c$ contains all but finitely many vertices
 from a ray $R$ then $c$ also contains also all but
finitely many vertices from
 any ray in the same end as $R$.  Thus we can say that the end
 belongs to $c$.  Define $\Omega c$ as the set of ends that belong to $c$.
The topology on the
 space of ends has as a basis the sets $\Omega c$ where $c\subset
 V\Gamma$ and $|\delta c|<\infty$.
The graph $\Gamma\setminus\delta c$ is not connected.  If
 $\omega_1$ and $\omega_2$ are two ends of $\Gamma$ and $\omega_1$ belongs
 to $c$ and $\omega_2$ belongs to $V\Gamma \setminus c$ then we can say that
 $\delta c$ separates $\omega_1$ and $\omega_2$.

It is easy to show that a quasi-isometry between two locally finite
connected graphs induces a homeomorphism between the end spaces
of the graphs.  In particular the number of ends is a quasi-isometry
invariant.

For a detailed introduction to ends of graphs the reader can
consult \cite{Moller1995} or \cite{DiestelKuhn2003}.

\subsubsection{Stallings' Theorem}
For a finitely generated group the number of ends is defined as the
number of ends of a Cayley graph of the group with respect to a finite
generating set. This is well defined, because the number of ends of a space is
invariant under quasi-isometry, as noted above.  It can be shown
that a finitely generated group has  $0$, $1$, $2$ or $\infty$ many ends.

Let $G$ be a compactly generated, totally disconnected, locally
compact group. Since $G$ admits a rough Cayley graph $\Gamma$,
which is unique up to quasi-isometry, we may define the
\emph{space of rough ends} of $G$ as the space of ends of $\Gamma$.
The cardinality of the space of rough ends of $G$ will be called the
\emph{number of rough ends} of $G$. These definitions obviously
extend the traditional concepts for finitely generated groups.
Because of Lemma~\ref{Ltransitive}, every compactly generated, totally
disconnected, locally compact group has  $0$, $1$, $2$ or
$\infty$~many ends.

A compactly generated, totally disconnected, locally compact group
has $0$~rough ends, if and only if it is compact (in particular, a finitely
generated group has $0$~ends if and only if it is finite).

Finitely generated groups with precisely two ends are characterized by
the following
result, which is a conjunction of results of Hopf and C.~T.~C.~Wall.

\begin{theorem}\label{TTwoEnds}
{\rm (\cite[Satz 5]{Hopf1944} and \cite[Lemma~4.1]{Wall1967})}
Let $G$ be a finitely generated group.  Then the following
are equivalent:

(i)  $G$ has precisely two ends;

(ii)  $G$ has an infinite cyclic subgroup of finite index;

(iii)  $G$ has a finite normal subgroup $N$ such that $G/N$ is
either isomorphic to the infinite cyclic group or to the infinite
dihedral group.
\end{theorem}

For compactly generated, totally disconnected, locally compact group we
have the following analogue.

\begin{theorem}\label{TStructertree_two} {\rm
    (Cf. \cite{MollerSeifter1998})}
Let $G$ be a compactly generated, totally disconnected, locally compact
group.  Suppose that the space of rough ends has precisely two
points. Then $G$ has a compact, open, normal subgroup $N$ such that
$G/N$ is either isomorphic to the infinite cyclic group or to the infinite
dihedral group.
\end{theorem}

Finitely generated groups with more than one end are described in a 
famous result of Stallings from 1968.

\begin{theorem}{\rm \cite{Stallings1968}}
Suppose $G$ is a finitely generated group with more than one end.
Then $G$ can be written as a non-trivial free product with amalgamation
$A*_CB$ (with $A\neq C\neq
B$)  where $C$ is finite, or $G$ can be written as a non-trivial
HNN-extension $A*_Cx$  where $C$ is finite.
\end{theorem}

The converse also holds, if $A*_CB$ (with $A\neq C\neq
B$)  where $C$ is finite, or $G$ can be written as a non-trivial
HNN-extension $A*_Cx$  where $C$ is finite, then $G$
has more than one end.

In 1974 Abels \cite{Abels1974} proved an analogue of Stallings' Theorem for
topological groups.  Abels uses similar ideas in his proof as used by
Stallings.  Essentially the same result as in \cite{Abels1974} is proved in
\cite{KronMoller2008}, using Dunwoody's theory of structure trees (see
\cite{Dunwoody1982} and \cite{DicksDunwoody1989}) and the Bass-Serre
theory of group actions on trees.

\begin{theorem}\label{TStructertree_infinite}
{\rm (Cf.~\cite{Abels1974})}
Let $G$ be a compactly generated, totally disconnected, locally compact
group.  Suppose that some (equivalently, any) rough Cayley graph of $G$
has infinitely many ends.   Then  $G=A*_C B$ (with $A\neq C\neq
B$) or $G=A*_C x$, where
$A$ and $B$ are open, compactly generated subgroups of $G$, and
$C$ is a compact, open subgroup.
\end{theorem}

That it is possible to write $G$ as either a free product with
amalgamation $G=A*_C B$ or an HNN-extension $G=A*_C x$ is often expressed by
saying that $G$ splits over $C$. Using this expression, Stallings
theorem becomes the statement that a finitely generated group with
infinitely many ends splits over a finite subgroup, and
Theorem~\ref{TStructertree_infinite} above becomes the statement that
a compactly generated, totally disconnected, locally compact group with
more than one end splits over a compact, open subgroup.
If $G$ is  totally disconnected, locally compact group with closed,
cocompact subgroup $H$, then a rough Cayley graphs for $G$
and a rough Cayley graph for  
$H$ are quasi-isometric by Theorem~\ref{TCocompact}. In particular,
$G$ and $H$ have the same number of rough ends. Hence,
Theorem~\ref{TStructertree_infinite} has the following corollary.

\begin{corollary}{\rm (\cite[Corollary~3.22]{KronMoller2008})}
Let $G$ be a totally disconnected, locally compact group and $H$ a
closed, cocompact subgroup.  Then $G$ splits over a compact, open
subgroup if and only if $H$ splits over a compact, open subgroup.
\end{corollary}

\subsubsection{Free subgroups}
A well known theorem of Gromov (also proved by Woess
\cite{Woess1989a}) says that a finitely generated group is
quasi-isometric to a tree if and only if it has a finitely generated
free subgroup of finite index.

The next theorem provides an analogue of this result for
compactly generated, totally disconnected, locally compact groups.

 \begin{theorem}{\rm (\cite[Theorem~3.28]{KronMoller2008})}
\label{TQuasiTree}
 Let $G$ be a compactly generated, totally disconnected,
  locally compact group.

(i)
Some (hence, every) rough Cayley graph of $G$ is
quasi-isometric to a tree if and only if $G$ has an expression as a
fundamental group of a finite graph of groups such that all the vertex and edge
groups are compact open subgroups of $G$.

(ii)
Assume also that the group $G$ is unimodular.  Then
some (hence, every) rough Cayley graph of $G$ is
quasi-isometric to some tree if and only if $G$ has a finitely
generated free subgroup that is cocompact and discrete.
\end{theorem}

The proof of Theorem \ref{TQuasiTree} uses information about ends,
quasi-isometries and structure trees from \cite{Woess1989a}, \cite{Kron2001a},
\cite{KronMoller2008a}, \cite{Moller1992a} and
\cite{ThomassenWoess1993}.  The essential result used about graphs
quasi-isometric to trees is that a transitive, locally finite, graph is
quasi-isometric to a tree if and only if it has no {\em thick ends}
(cf.~\cite{Woess1989a} and \cite[Theorem~5.5]{KronMoller2008a}).  This property can then be
used to show that for a
locally finite, transitive graph that is quasi-isometric to a tree
one can find a locally finite structure tree on which the automorphism
group of the original graph acts.

For the proof of existence of a cocompact, discrete, finitely
generated, free subgroup in part~(\romannumeral2) of
Theorem~\ref{TQuasiTree} the theory of tree lattices in~\cite{BassKulkarni1990} is used.

\begin{corollary} {\rm (\cite[Corollary~3.29]{KronMoller2008})}
Let $G$ be a totally disconnected,
  locally compact group.  If $G$ has a cocompact, finitely generated,
  free, discrete subgroup, then $G$ splits over some compact, open
  subgroup and $G$ can be written as $G=A*_C B$ (with $A\neq C\neq
B$) or $G=A*_C x$  where
$A, B$ and $C$ are compact, open subgroups of $G$.
\end{corollary}

The above Corollary implies a special case of a result
of Mosher, Sageev and Whyte
\cite[Theorem~9]{MosherSageevWhyte2003}.

\begin{corollary} {\rm (\cite[Corollary~3.30]{KronMoller2008})}
 Let $G$ be a  totally disconnected,
  locally compact group.  If $G$ has a cocompact, finitely generated,
  free, discrete subgroup, then $G$ has an action on a locally finite
  tree, such that $G$ fixes neither an edge nor a vertex.
\end{corollary}

Consider a finitely generated group $H$ and a Cayley graph $\Gamma$ of
$H$.  The action of $H$ on the Cayley graph gives an embedding of $H$
as a closed, cocompact subgroup into the totally disconnected, locally
compact group $G=\autga$.
Willis asks in \cite[Section 6]{Willis2004} whether various invariants
of $G$ can be bounded in terms of $H$.  For example, he asks if
it is possible to deduce that there are only finitely many
prime numbers that occur as factors in ${\bf s}(x)$ for $x\in G$
(the scale function ${\bf s}$ is discussed in  Section~\ref{STidyScale}).
This question is motivated by the following result.

\begin{theorem}{\rm (\cite[Theorem 3.4]{Willis2001})}
Let $G$ be a compactly generated, totally disconnected, locally compact
group.  Then there are finitely many primes $p_1, p_2, \ldots, p_t$ such that
the number ${\bf s}_G(x)$ for all elements $x\in G$ can be written in the
form ${\bf s}_G(x)=p_1^{s_1}p_2^{s_2}\cdots  p_t^{s_t}$.
\end{theorem}

Baumgartner \cite{Baumgartner2007} has applied the above mentioned result of
Mosher, Sageev and Whyte to the program suggested by Willis.
Baumgartner has also extended the scope of the program, by
considering not only the special type of embedding $G\rightarrow
\autga$, where $\Gamma$ is a Cayley graph of $H$.
A topological group $G$ is called an {\em envelope} of
a group $H$, if $H$ embeds as a closed, cocompact subgroup of $G$.

\begin{theorem}{\rm \cite[Corollary 11]{Baumgartner2007})}
Let $H$ be a virtually free group of rank at least 2.  Then there are
finitely many primes $p_1, p_2, \ldots, p_t$ such that
 for all totally disconnected, locally
compact envelopes $G$ of $H$ all elements $x\in G$
the number ${\bf s}_G(x)$ can be written in the
form ${\bf s}_G(x)=p_1^{s_1}p_2^{s_2}\cdots  p_t^{s_t}$.
\end{theorem}

The following result comes from a totally different direction, but it also
indicates the possible fruitfulness of the
study of envelopes and embeddings of a finitely
generated group into automorphism groups of its Cayley graphs.

\begin{theorem}{\rm \cite[Theorem~4.1]{MollerSeifter1998})}
Let $N$ be a finitely generated, torsion free, nilpotent group and
$\Gamma$ a Cayley graph of $N$ with respect to some finite generating
set of $N$. Put $G=\autga$.  Then $G$ is discrete in the permutation
topology, and $N$ embeds into $G$ as normal subgroup.
\end{theorem}

This theorem is proved with the aid of Theorem~\ref{TTrofimovoauto}.

\subsubsection{Accessibility}
\begin{definition}\label{DAccessibility}
A finitely generated group is said to be {\em accessible} if it has an
action on a tree $T$ such that:

{(i)} the number of orbits of $G$ on the edges of $T$ is finite;

{(ii)} the stabilizers of edges in $T$ are finite subgroups of
$G$;

{(iii)} every stabilizer of a vertex in $T$ is a finitely
generated subgroup of $G$ with at most one end.
\end{definition}

The question by C.~T.~C.~Wall, \cite{Wall1971}, of whether or not every
finitely generated group is accessible  motivated several important
developments in combinatorial group theory, among them
Dunwoody's theory of structure trees. Wall's question,  after being
open for a long time, was settled by examples of finitely generated
groups that are not accessible, which were constructed by Dunwoody \cite{Dunwoody1991}.

\begin{definition}{\rm (\cite[p.~249]{ThomassenWoess1993})}
\label{DGraphAccessibility}
Let $\Gamma$ be a connected, locally finite graph.  If there is a
number $k$ such that any two distinct ends can be separated by
removing $k$ or fewer vertices from $\Gamma$ then the graph $\Gamma$
is said to be {\em accessible}.
\end{definition}

Thomassen and Woess
\cite[Theorem~1.1]{ThomassenWoess1993} show that a finitely generated
  group is accessible if and only if every Cayley graph with respect
  to a finite generating set is accessible.

The notion of accessibility can be generalized to compactly generated,
totally disconnected, locally compact groups as follows.

\begin{definition}\label{DTopoAccessibility}
A compactly generated, totally disconnected, locally compact
group is said to be {\em accessible} if it has an
action on a tree $T$ such that:

(i) the number of orbits of $G$ on the edges of $T$ is finite;

(ii) the stabilizers of edges in $T$ are compact, open subgroups of
$G$;

(iii) every stabilizer of a vertex in $T$ is a compactly
generated, open subgroup of $G$ with at most one rough end.

\end{definition}

Then one can link accessibility 
of compactly generated totally disconnected locally
compact groups to rough Cayley graphs in analogues
way as Thomassen and Woess link accessibility of finitely generated
groups to Cayley graphs.

 \begin{theorem}{\rm (\cite[Theorem~3.27]{KronMoller2008})}
\label{TAccessible}
Let $G$ be a compactly generated, totally disconnected, locally compact
group.  Then $G$ is accessible if and only if every rough Cayley graph
of $G$ is accessible.
\end{theorem}

In fact, the group in Theorem \ref{TAccessible} is accessible if and only if
any of its rough Cayley graphs is accessible, because the
property of a graph being accessible is invariant under
quasi-isometries by~\cite[Theorem~0.4]{PapasogluWhyte2002}.  By the same result,
a compactly generated, totally disconnected, locally compact group
with a closed, cocompact, accessible subgroup is itself accessible.
Since every finitely presentable group is accessible by a result
of Dunwoody \cite{Dunwoody1985}, we deduce the following theorem as a corollary.

\begin{theorem}
Let $G$ be a compactly generated, totally disconnected, locally compact
group.  If $G$ has a cocompact, finitely presented subgroup then $G$ is
accessible.
\end{theorem}

\subsubsection{Ends of pairs of groups}
By Stallings' Theorem the ends of a finitely generated group can be
used to detect if the
group splits over a finite subgroup.  The concept of ends of pairs of
groups is an attempt to define a geometric invariant that can be used
to detect splittings of $G$ over subgroups that are not 
finite.  This concept was first introduced in the papers \cite{Houghton1974}
by Houghton and \cite{Scott1977} by Scott.
For a survey of these
and related concepts see \cite{Wall2003}.

\begin{definition}{\rm (Cf.~\cite[Lemma~1.1]{Scott1977})}  Let $G$ be a finitely
  generated group and $C$ a subgroup of $G$.  Let $\Gamma$ be a
  Cayley graph of $G$ with respect to some finite generating set.  The
  {\em number of ends of the pair} $(G,C)$, denoted with $e(G,C)$,  
is defined as the number of ends
  of the quotient graph $C\backslash\Gamma$  (quotient with respect to the
  natural $C$-action on $\Gamma$).  
\end{definition}

It can be shown that the number of ends of a pair of groups does not
depend on the choice of generating set. While a transitive, locally finite
graph has $0$, $1$, $2$ or infinitely many ends, a pair of groups can
have any number of ends, see~\cite[Example~2.1]{Scott1977}.

The following conjecture generalizing Stallings' Theorem
is due to Kropholler, see \cite{NibloSageev2006}.

\begin{conjecture}
Let $G$ be a finitely generated group and $C$ a subgroup of $G$.  If 
$G$ contains a subset $A$ such that 

(i) $A=CA$;

(ii) for every element $g\in G$ the symmetric difference of $a$ and
$Ag$ is contained 
in a finite union of right $C$ cosets;

(iii)  neither $A$ nor $G\setminus A$ is contained in any finite union
of  right $C$ cosets;

(iv) $A=AC$

\noindent
then $G$ splits over a subgroup that is commensurable
with a subgroup of $C$.  (Conditions (i)-(iii) above are equivalent to
$e(G,C)\geq 2$.)
\end{conjecture}

Here, two subgroups are are said to be \emph{commensurable},
if their intersection has finite index in both subgroups. Furthermore the
\emph{commensurator} of a subgroup $C$ of $G$ is the set
of elements $g\in G$ such that $C$ and $gCg^{-1}$ are commensurable.
The commensurator of a subgroup is itself a subgroup.
Kropholler's conjecture above has been verified under various
additional hypotheses.  A sample result is the following theorem.

\begin{theorem}{\rm (\cite[p.~30]{DunwoodyRoller1993})}\label{Tcommensurator}
 Let $G$ be a finitely generated group and $C$ a finitely generated
  subgroup of $G$.  If $e(G, C)>1$  and
the commensurator of $C$ is the whole group $G$ then $G$
  splits over a subgroup commensurable with $C$.
\end{theorem}
This result has also been proved in papers
by Niblo \cite[cf. Theorem~B]{Niblo2002} and Scott and Swarup
\cite[Theorem~3.12]{ScottSwarup2000}.

In \cite[Section~3.7]{KronMoller2008} there is further discussion of how
the concepts of ends of pairs of groups, coends and rough ends relate
and how these concepts can be interpreted graph theoretically.
Amongst other things these considerations lead to a prove of
Theorem~\ref{Tcommensurator} above.

\subsection{Polynomial growth}\label{SPolynomial}
Recall from Section~\ref{STrofimov} that
a connected, locally finite graph is said to have {\em polynomial
  growth} if for every vertex the number of vertices in distance less
  than or equal to $n$ grows polynomially with $n$.  A finitely
  generated group is said to have polynomial growth if some (hence,
  every) Cayley graph with respect to a finite generating set has
  polynomial growth.

The concept of polynomial growth can be generalized to compactly
generated, locally compact groups.

\begin{definition}\label{DTopoPolynomial}
Let $G$ be a locally compact group generated by a compact
symmetric neighbourhood of the identity $V$.  Set
$V^n=\{g_1 g_2\cdots g_n\mid g_i\in V\}$.
Let $\mu$ denote a Haar measure on
$G$.  If there are constants $c$ and $d$ such that $\mu(V^n)\leq cn^d$
for all natural numbers $n$, $G$ is said to have {\em polynomial growth}.
\end{definition}

The following theorem characterizes compactly generated,
totally disconnected groups of polynomial growth in terms of
their rough Cayley graphs.

\begin{theorem}{\rm (\cite[Theorem~4.4]{KronMoller2008})}
\label{TopoPolynomial}
Let $G$ be a compactly generated, totally disconnected,
locally compact group and $\Gamma$
some rough Cayley graph for $G$.  Then $\Gamma$ has polynomial growth
if and only if $G$ has polynomial growth (in the sense of
Definition~\ref{DTopoPolynomial}).
\end{theorem}

Viewed in this context, Trofimov's theorem about automorphism groups of
graphs with polynomial growth, Theorem~\ref{TTrofimovPoly}, can now be
seen as a version of Gromov's Theorem for compactly generated, totally
disconnected, locally compact groups.
As already mentioned,  Losert  proved a generalization
of Gromov's Theorem for topological groups in \cite{Losert1987} and
Woess deduced Trofimov's theorem from Lostert's result
in \cite{Woess1992}.
The following theorem \cite[Theorem~4.6]{KronMoller2008} is a
combination of Theorem~\ref{TopoPolynomial} and Trofimov's theorem, but
can also be seen as Corollary to Losert's results.

\begin{theorem} {\rm (Cf.~Trofimov \cite{Trofimov1985} and Losert
    \cite{Losert1987})}
Let $G$ be a compactly generated, totally disconnected, locally compact
group.  Then $G$ has polynomial growth if and only if $G$ has a normal,
compact, open
subgroup $K$ such that $G/K$ is a finitely generated almost nilpotent group.
\end{theorem}


\frenchspacing
\begin{thebibliography}{99}

\bibitem{Abels1974}  H.\ Abels,
                Specker-Kompaktifizierungen von lokal kompakten
                topologischen Gruppen.
                {\em Math.\ Z.} {\bf 135} (1974), 325--361.

\bibitem{AbertGlasner2008} M.~Abert and Y.~Glasner,
             Generic groups acting on regular trees.
             {\em Trans. Amer. Math. Soc.} {\bf 361} (2009), 3597--3610.

\bibitem{BassKulkarni1990}
                H.~Bass and R.~Kulkarni,
                Uniform tree lattices.
                {\em J. Amer. Math. Soc.} {\bf 3} (1990), 843--902.

\bibitem{BassLubotzky2001} H.~Bass and A.~Lubotzky,
               {\em Tree lattices}.
               Progress in Mathematics, 176. Birkh\"auser Boston,
               Boston, 2001.

\bibitem{Baumgartner2007}  U.~Baumgartner,
                Scales for co-compact embeddings of virtually free
                groups.
                {\em Geom. Dedicata}, {\bf 130}  (2007), 163--175.

\bibitem{BaumgartnerWillis2006}  U.~Baumgartner and G.~Willis,
                The direction of an automorphism of a totally
                disconnected locally compact group.
                {\em  Math. Z.} {\bf 252}  (2006),  393--428.

\bibitem{Bhattacharjee1994}  M.~Bhattacharjee,
                 The probability of generating certain profinite groups by
                 two elements.
                 {\em Israel J. Math.} {\bf 86} (1994),  311--329.

 \bibitem{Bhattacharjee1995}  M.~Bhattacharjee,
               The ubiquity of free subgroups in certain inverse limits
               of groups.
               {\em J. Algebra}  {\bf 172}  (1995), 134--146.

   \bibitem{BMMN1997} M.\ Bhattacharjee, D.\ Macpherson, R.\ G.\
                M{\"o}ller and P.\ M.\ Neumann,
                {\em Notes on Infinite Permutation Groups}.
                Hindustan Book Agency, Delhi, India, 1997.
                (Republished  as Springer Lecture Note
                in Mathematics 1698, Springer, 1998.)

\bibitem{BergmanLenstra1989} G.\ M.\ Bergman and H.\ W.\ Lenstra,
                Subgroups close to normal subgroups.
                {\em J.\ Algebra} {\bf 127} (1989), 80--97.

\bibitem{Cameron1990} P.\ J.\ Cameron,
{\em Oligomorphic permutation groups}.
London Mathematical Society Lecture Note Series 152,
Cambridge University Press, Cambridge, 1990.

\bibitem{Cameron1999} P.\ J.\ Cameron,
{\em Permutation groups}.
London Mathematical Society Student Texts 45,
Cambridge University Press, Cambridge, 1999.

\bibitem{CPW1993}  P.\ J.\ Cameron, C.\ E.\ Praeger and N.\ C.~Wormald,
                Infinite highly arc transitive digraphs and universal
                covering digraphs.
                {\it Combinatorica} {\bf 13} (1993), 377--396.

\bibitem{Dantzig1936} D.\ van Dantzig,
                Zur topologischen Algebra III.  Brouwersche und
                Cantorsche Gruppen.
                {\em Compos.\ Math.} {\bf 3}  (1936), 408--426.

\bibitem{DicksDunwoody1989} W.~Dicks and M.~J. Dunwoody,
                {\em Groups acting on graphs}.
                Cambridge University Press, Cambridge, 1989.

\bibitem{DiestelKuhn2003} R.~Diestel, D.~K\"uhn,
                Graph-theoretical versus topological ends of graphs.
                {\em J. Combin. Theory Ser. B} {\bf 87} (2003),
                197--206.

\bibitem{DiestelLeader2001} R.\ Diestel and I.\ Leader,
                A conjecture concerning a limit of non-Cayley graphs.
                {\em J.~Algebraic Combin.} {\bf 14} (2001), 17--25.

\bibitem{DixonMortimer1996}  J.\ D.\ Dixon and B.\ Mortimer,
                {\em Permutation Groups}.
                Graduate text in mathematics 163, Springer 1996.

\bibitem{Dunwoody1982} M.~J.~Dunwoody,
                Cutting up graphs.
                {\em Combinatorica} {\bf 2} (1982), 15--23.

\bibitem{Dunwoody1985} M.~J.~Dunwoody,
                The accessibility of finitely presented groups.
                {\em Invent. Math.} {\bf 81} (1985), 449--457.

\bibitem{Dunwoody1991} M.~J.~Dunwoody,
                An inaccessible group.  In
                {\em Geometric Group Theory 1991, Vol. 1}
                (ed. by G.~A.~Niblo and M.~A.~Roller),
                75--78,
                L.M.S. Lecture Notes Series 181, Cambridge University
                Press, Cambridge, 1993.

\bibitem{DunwoodyRoller1993} M.\ J.\ Dunwoody and M.\ A.\ Roller
                Splitting groups over polycyclic-by-finite subgroups.
                {\em Bull. London Math. Soc.} {\bf 25} (1993), 29--36.

\bibitem{EskinFisherWhyte2005} A.~Eskin, D.~Fisher, K.~Whyte,
                Quasi-isometries and rigidity of solvable groups.
                {\em Pure Appl.\ Math.\ Q.} {\bf 3} (2007), 927-947.

\bibitem{EskinFisherWhyte2006} A.~Eskin, D.~Fisher, K.~Whyte,
                Coarse differentiation of quasi-isometries I: spaces
                not quasi-isometric to Cayley graphs.
                preprint 2006, {\tt arXiv:math/0607207v2}.

 \bibitem{Evans1997}  D.\ M.\ Evans, An infinite highly arc-transitive
                digraph. {\it Europ.\ J.\ Combin.} {\bf 18} (1997), 281--286.

 \bibitem{Evans2001}  D.\ M.\ Evans, Suborbits in infinite primitive
                permutation groups. {\em Bull. London Math. Soc.}
                {\bf 33} (2001), 583--590.

\bibitem{Figa-TalamancaNebbia1991}A.~Fig{\`a}-Talamanca and
C.~Nebbia, {\em Harmonic analysis and representation theory for groups acting on homogeneous trees}.
London Mathematical Society Lecture Note Series, 162. Cambridge University Press, Cambridge, 1991.

\bibitem{Freudenthal1942} H. Freudenthal,
                Neuaufbau der Endentheorie.
                {\em Ann. of Math. (2)} {\bf 43} (1942), 261--279.

\bibitem{Freudenthal1945} H. Freudenthal,
                \"Uber die Enden diskreter R\"aume und Gruppen.
                {\em Comm. Math. Helv.} {\bf 17} (1945), 1--38.

\bibitem{Gromov1981}
                M.\ Gromov,
                Groups of polynomial growth and expanding maps,
                {\em Publ.\ Math.\ IHES}, {\bf 53} (1981), 53--78.

\bibitem{Gromov1986}
                M.\ Gromov,
                Infinite groups as geometric objects. 
                In {\em Proceedings of the International Congress of
                Mathematicians Vol. 1,}  385--392,
                PWN, Warsaw, 1984. 

\bibitem{GrosserMoskowitz1971}  S.\ Grosser and M.\ Moskowitz,
                Compactness conditions in topological groups.
                {\em J.\ Reine Angew. Math.} {\bf 246} (1971), 1--40.

\bibitem{Halin1964} R.~Halin,
                \"Uber unendliche Wege in Graphen.
                {\em Math. Ann.} {\bf 157} (1964), 125--137.

\bibitem{Hopf1944}  H.~Hopf,
                Enden offener R\"aume und unendliche diskontinuierlich
                Gruppen.
                {\em   Comment. Math. Helv.} {\bf 16}  (1944),
                81--100.

\bibitem{Houghton1974} C.\ H.\ Houghton,
                Ends of locally compact groups and their coset spaces.
                {\em J.\ Austral.\ Math.\ Soc.} {\bf 17} (1974), 274--284.

\bibitem{Ivanov1983} A. A. Ivanov,
                Bounding the diameter of a distance-regular graph.
                {\em Soviet Math. Doklady} {\bf 28} (1983), 149--152.
                (Translation from {\em Dokl. Akad. Nauk SSSR} {\bf  271} (1983),  789--792.)

\bibitem{Iwasawa1951} K.~Iwasawa, Topological groups with
                invariant compact neighborhoods of the identity.
                {\em Ann. of Math. (2)} {\bf  54} (1951),  345--348.

\bibitem{KarrassSolitar1956} A.~Karrass and D.~Solitar,
                Some remarks on the infinite symmetric groups.
                {\em Math. Z.} {\bf 66} (1956), 64--69.

\bibitem{Kron2001a} B.~Kr\"on,
                Quasi-isometries between non-locally-finite graphs
                and structure trees.
                {\em Abh.\ Math.\ Sem.\ Univ.\ Hamburg} {\bf
                71}(2001), 161--180.

\bibitem{KronMoller2008} B.~Kr\"on and R.~G.~M\"oller,
                Analogues of Cayley graphs for topological groups,
                {\em Math. Z.} {\bf 258} (2008), 637-675.

\bibitem{KronMoller2008a} B.~Kr\"on and R.~G.~M\"oller,
                Quasi-isometries between graphs and trees.
                {\em J.~Combin. Theory Ser. B} {\bf 98} (2008), 994--1013.

\bibitem{Losert1987}  V.~Losert,
             On the structure of groups with polynomial
  growth.
                {\em Math. Z.} {\bf 195} (1987), 109--117.

\bibitem{Losert2002} V.~Losert,
           On the structure of groups with polynomial growth II.
           {\em J.~London Math.~Soc. (2)} {\bf 63} (2001), 640--654.

\bibitem{Macpherson1982} H.\ D.\ Macpherson,
                Infinite distance transitive graphs of finite valency.
                {\em Combinatorica} {\bf 2} (1982), 63--69.

\bibitem{MMMSTZ2005}  A.\ Malni\v{c}, D.\ Maru\v{s}i\v{c},
R.~G.~M\"oller, N.\ Seifter, V.\ Trofimov and B.\ Zgrabli\v{c},
Highly arc transitive digraphs: reachability,
  topological groups.  {\it European J. Combin.} {\bf
    26} (2005), 19--28.

\bibitem{Maurer1955}  I.~Maurer,
                 Les groupes de permutations infinies.
                 {\em Gaz. Mat. Fiz. Ser. A.} {\bf 7} (1955), 400--408.

\bibitem{Moller1992a} R.~G.~M\"oller,
                Ends of graphs II.  {\em Math.\ Proc.\ Camb.\ Phil.\ Soc}.
                {\bf 111} (1992), 455--460.

\bibitem{Moller1995} R.~G.~M\"oller,
                Groups acting on locally finite graphs --- a
                survey of the infinitely ended case.
                In {\em Groups'93 Galway/St Andrews, Vol. 2}
                (ed. by C.~M.~Campbell, T.~C.~Hurley, E.~F.~Robertson,
                S.~J.~Tobin and J.~J.~Ward),  426--456,
                L.M.S.~Lecture Notes Series 212,
                Cambridge University Press, Cambridge, 1995.

\bibitem{Moller1998} R.~G.~M\"oller,
                Topological groups, automorphisms of infinite graphs and
                a theorem of Trofimov.
                {\em Discrete Math.} {\bf 178} (1998), 271--275.

\bibitem{Moller2002} R. G. M\"oller,
                Structure theory of totally disconnected locally
                compact groups via graphs and permutations.
                {\em Canad.~J.~Math.} {\bf 54} (2002), 795--827.

\bibitem{Moller2002a} R. G. M\"oller,
                Descendants in highly arc transitive digraphs.
                {\em Discrete Math.} {\bf 247} (2002), 147--157.
                (Erratum. {\em Discrete Math.} {\bf 260} (2003),  321.)

\bibitem{Moller2003} R.\ G.\ M\"oller,
                FC$^-$-elements in totally disconnected groups and
                automorphisms of infinite graphs.
                {\em Math. Scand.} {\bf 92} (2003), 261--268.


\bibitem{MollerSeifter1998} R.~G.~M\"oller and N.~Seifter,
                 Digraphical regular representations of infinite
                 finitely generated groups.
                 {\em European J. Combin.} {\bf 19}  (1998), 597--602.

\bibitem{MosherSageevWhyte2003} L.~Mosher, M.~Sageev, K.~Whyte,
                Quasi-actions on trees. I. Bounded valence.
                {\em Ann. of Math. (2)}, {\bf 158} (2003), 115-164.

\bibitem{Nebbia2000} C.\ Nebbia,
                Minimally almost periodic totally disconnected groups.
                {\em Proc.\ Amer.\ Math.\ Soc.} {\bf 128} (2000), 347--351.

\bibitem{Niblo2002} G.\ A.\ Niblo,
                The singularity obstruction for group splittings.
                {\em Topology Appl.}  {\bf 119} (2002), 17--31.

\bibitem{NibloSageev2006}  G.\ A.\ Niblo and M.~Sageev,
The Kropholler conjecture.  In {\em Guido's book of conjectures}
(ed. by Indira Chatterji), 2006.\\
{\tt http://www.math.ohio-state.edu/$\sim$indira/GMFinal.pdf}.

\bibitem{PapasogluWhyte2002}  P.~Papasoglu, K.~Whyte,
                Quasi-isometries between groups with infinitely many ends.
                {\em Comment. Math. Helv.} {\bf  77} (2002),
                133--144.

\bibitem{Praeger1991} C. E. Praeger,
                On homomorphic images of edge transitive directed graphs.
                {\em Australas. J. Combinatorics} {\bf 3} (1991), 207--210.

\bibitem{Sabidussi1964} G.\ Sabidussi,
                Vertex-transitive graphs.
                {\em Monatsh. Math.} {\bf 68} (1964),  426--438.

\bibitem{Schlichting1979} G.\ Schlichting,
                Polynomidentit\"aten und Permutationsdarstellungen
                lokalkompakter Gruppen.
                {\em Invent.\ Math.} {\bf 55} (1979), 97--106.

\bibitem{Schlichting1980} G.\ Schlichting,
                Operationen mit periodischen Stabilisatoren.
                {\em Arch.\ Math.} {\bf 34} (1980), 97--99.

\bibitem{Scott1977} P.~Scott,
                Ends of pairs of groups.
                {\em J. Pure Appl. Algebra} {\bf  11} (1977/78), 179--198.

\bibitem{ScottSwarup2000} P.~Scott and G.~A.~Swarup,
                Splittings of groups and intersection numbers.
                {\em Geom.~Topol.} {\bf 4} (2000), 179--218.

\bibitem{Stallings1968} J.~R.~Stallings,
                On torsion free groups with infinitely many ends.
                {\em Ann. of Math.} {\bf 88} (1968), 312--334.

\bibitem{Tits1964}  J.\ Tits,
                Automorphismes {\`a} d{\'e}placement born{\'e}
                des groupes  de Lie.
                {\em Topology} {\bf 3} {\em Suppl.\ 1} (1964), 97--107.

\bibitem{ThomassenWoess1993} C.~Thomassen and W.~Woess,
                Vertex--transitive graphs and accessibility.
                {\em J.~Combin. Theory Ser. B} {\bf 58} (1993), 248--268.

\bibitem{Trofimov1984}  V.~I.~Trofimov,
                Automorphisms of graphs and a characterization of lattices.
                {\em Math USSR Izv.} {\bf 22} (1984), 379--392.
                 (Translation from {\em Izv.~Akad.~Nauk SSSR Ser.~Mat.} {\bf 47} (1983),  407--420.)

\bibitem{Trofimov1985}  V.~I.~Trofimov,
                Graphs with polynomial growth.
                {\em Math USSR Sb.} {\bf 51} (1985), 405--417.
(Translation from {\em Mat.~Sb.~(N.S.)} {\bf 123} (1984),  407--421.)

\bibitem{Trofimov1985a}  V. I. Trofimov,
                Automorphism groups of graphs as topological groups.
                {\em Math. Notes} {\bf 38} (1985), 717--720.
 (Translation from {\em Mat.~Zametki} {\bf 38} (1985), 378--385, 476.)

\bibitem{Trofimov1987}  V.\ I.\ Trofimov,
                The action of a group on a graph.
                {\em Math USSR Izv.} {\bf 29} (1987), 429--447.
           (Translation from {\em Izv. Akad. Nauk SSSR Ser. Mat.} {\bf 50} (1986), 1077--1096.)

\bibitem{Trofimov2007}  V.\ I.\ Trofimov,
                Vertex stabilizers of graphs and tracks. I.
                {\em European J. Combin.} {\bf 28}  (2007), 613--640.

\bibitem{Wall1967} C.~T.~C.~Wall,
                Poincar{\'e} complexes. I.
                {\em Ann. of Math.} (2) {\bf 86} (1967), 213--245.

\bibitem{Wall1971} C.~T.~C.~Wall,
                Pairs of relative cohomological dimension one.
                {\em J.~Pure Appl. Algebra} {\bf 1} (1971), 141--154.

\bibitem{Wall2003} C.~T.~C.~Wall,
                The geometry of abstract groups and their splittings.
                {\em Rev. Mat. Complut.}  {\bf 16}  (2003),   5--101.

\bibitem{Willis1994} G.\ Willis,
                The structure of totally disconnected, locally compact
                groups.
                {\em Math. Ann.} {\bf 300} (1994), 341--363.

\bibitem{Willis1995} G.\ Willis,
                Totally disconnected groups and proofs of conjectures
                of Hofmann and Mukherjea.
                {\em Bull.\ Austral.\ Math.\ Soc.} {\bf 51} (1995),
                489--494.

\bibitem{Willis2001}  G.~Willis,
                The number of prime factors of the scale function on a
                compactly generated group is finite.
                {\em Bull. London Math. Soc.} {\bf 33} (2001),
                168--174.

\bibitem{Willis2001a} G.\ Willis,
                Further properties of the scale function on a totally
                disconnected group.
                {\em J.~Algebra} {\bf 237} (2001), 142--164.

\bibitem{Willis2004} G.\ Willis,
                A canonical form for automorphisms of totally
                disconnected locally compact groups.
                in {\em Random walks and geometry}
                (ed. by V.~A.~Kaimanovich
                 in collaboration with K.~Schmidt and W.~Woess),
                295--316, Walter de Gruyter GmbH \& Co. KG,
                Berlin, 2004.

\bibitem{Wolf1968} J.~A.~Wolf,
               Growth of finitely generated solvable groups and
               curvature of Riemannian manifolds.
               {\em J.~Differential Geometry} {\bf 2} (1968),
               421--446.

\bibitem{Woess1989a} W.~Woess,
                Graphs and groups with tree-like properties.
                                {\em J.~Combin. Theory Ser. B} {\bf 47} (1989),
                361--371.

\bibitem{Woess1992} W. Woess,
                Topological groups and infinite graphs.  In
                {\em Directions in Infinite Graph Theory and
                Combinatorics}. (ed. by R. Diestel), Topics in Discrete
                Math. 3, North Holland, Amsterdam 1992.
                (Also in {\em Discrete Math.}  {\bf 95} (1991), 373--384.)

\bibitem{Woess2001} W. Woess,
              {\em Random walks on infinite graphs and
              groups}.  Cambridge Tracts in Mathematics, 138. Cambridge
              University Press, Cambridge, 2000.

\bibitem{WuYu1972} T.\ S.~Wu and Y.\ K.\ Yu,
               Compactness properties of topological groups.
               {\em Michigan Math.\ J.} {\bf 19} (1972), 299--313.

\bibitem{Wu1991}] T.\ S.~Wu,
                On the structure of certain locally compact topological groups.
                {\em Trans.\ Amer.\ Math.\ Soc.}  {\bf 325} (1991),
                413--434.

\end{thebibliography}
\end{document}